\documentclass[12pt,a4paper]{amsart}
\usepackage[utf8]{inputenc}
\usepackage[english]{babel}
\usepackage{graphicx}
\usepackage{amsmath}
\usepackage{amssymb}
\usepackage{amscd}
\usepackage{amsthm}
\usepackage{amscd}

\usepackage[matrix,arrow,curve]{xy}
\usepackage{stmaryrd}
\usepackage{hyperref}

\usepackage{tikz}
\usetikzlibrary{arrows.meta}% arrow tip library
\usetikzlibrary{bending}% better arrow head for bended lines

\usepackage{graphicx}
\usepackage{xspace}
\usepackage{enumerate}
\usepackage{verbatim}
\usepackage{geometry}
\usepackage{tikz-cd}

\usepackage{mathrsfs}

\usepackage[pagewise]{lineno}%\linenumbers

\DeclareMathOperator{\K}{\mathrm{K}}
\DeclareMathOperator{\F}{\mathsf{F}}

\begin{document}

\newcommand{\Sm}{\mathcal{S}^{\!}\mathsf{m}_k}
\newcommand{\Rings}{\mathcal{R}^{\!}\mathsf{ings}^*}
\newcommand{\rings}{\mathcal{R}^{\!}\mathsf{ings}}
\newcommand{\SmE}[1]{\mathcal{S}^{\!}\mathsf{m}_{#1}}
\newcommand{\Mot}[1]{\mathcal M^{\!}\mathsf{ot}_{#1}}
\newcommand{\Corr}[1]{\mathcal C^{\!}\mathsf{orr}_{#1}}
\newcommand{\MotF}[1]{\mathcal M_{#1}}
\newcommand{\QG}{\mathbb Q\Gamma}
\newcommand{\Inv}{\mathrm{Inv}(\QG)}
\newcommand{\Gal}{\mathrm{Gal}(k^{\mathrm{sep}}/k)}
\newcommand{\End}{\mathrm{End}}
\newcommand{\M}[1]{\mathcal{M}_{#1}}
\newcommand{\EG}{\!\,_EG}
\newcommand{\EGP}{\!\,_E(G/P)}
\newcommand{\EX}{\!\,_EX}
\newcommand{\XG}{\!\,_{\xi}G}
\newcommand{\XGP}{\!\,_{\xi}(G/P)}
\newcommand{\KQ}[1]{\mathrm K(n)^*\big(#1;\,\mathbb Q[v_n^{\pm1}]\big)}
\newcommand{\KZ}[1]{\mathrm K(n)^*\big(#1;\,\mathbb Z_{(p)}[v_n^{\pm1}]\big)}
\newcommand{\KZp}[1]{\mathrm K(n)^*\big(#1;\,\mathbb Z_p[v_n^{\pm1}]\big)}
\newcommand{\KF}[1]{\mathrm K(n)^*\big(#1;\,\mathbb F_p[v_n^{\pm1}]\big)}
\newcommand{\CHQ}[1]{\mathrm{CH}^*\big(#1;\,\mathbb Q[v_n^{\pm1}]\big)}
\newcommand{\KXZ}[1]{\!\,^{\mathrm K(n)\!}#1_{\,\mathbb Z_{(p)}[v_n^{\pm1}]}}
\newcommand{\KMotQ}{\Mot{\,\mathrm K(n)}}
\newcommand{\CHMotQv}{\Mot{\,\mathrm{CH}}}
\newcommand{\KXQ}[1]{\mathcal M_{\mathrm K(n)}(#1)}
\newcommand{\KMQ}{\mathcal M_{\,\mathrm K(n)}}
\newcommand{\CHMQv}{\MotF{\,\mathrm{CH}}}
\newcommand{\KCorrQ}{\Corr{\,\mathrm K(n)}}
\newcommand{\CHCorrQv}{\Corr{\,\mathrm{CH}}}
\newcommand{\CHCorrQ}{\Corr{\,\mathrm{CH}}}
\newcommand{\AMot}{\Mot A}

\newcommand{\e}{\varepsilon}
\newcommand{\con}{\ensuremath{\triangledown}}
\newcommand{\ra}{\ensuremath{\rightarrow}}
\newcommand{\tp}{\ensuremath{\otimes}}
\newcommand{\pr}{\ensuremath{\partial}}
\newcommand{\trigd}{\ensuremath{\triangledown}}
\newcommand{\dAB}{\ensuremath{\Omega_{A/B}}}
\newcommand{\QQ}{\ensuremath{\mathbb{Q}}\xspace}
\newcommand{\CC}{\ensuremath{\mathbb{C}}\xspace}
\newcommand{\RR}{\ensuremath{\mathbb{R}}\xspace}
\newcommand{\ZZ}{\ensuremath{\mathbb{Z}}\xspace}
\newcommand{\Zp}{\ensuremath{\mathbb{Z}_{(p)}}\xspace}
\newcommand{\Z}[1]{\ensuremath{\mathbb{Z}_{(#1)}}\xspace}
\newcommand{\NN}{\ensuremath{\mathbb{N}}\xspace}
\newcommand{\LL}{\ensuremath{\mathbb{L}}\xspace}
\newcommand{\inN}{\ensuremath{\in\mathbb{N}}\xspace}
\newcommand{\inQ}{\ensuremath{\in\mathbb{Q}}\xspace}
\newcommand{\inR}{\ensuremath{\in\mathbb{R}}\xspace}
\newcommand{\inC}{\ensuremath{\in\mathbb{C}}\xspace}
\newcommand{\OO}{\ensuremath{\mathcal{O}}\xspace}
\newcommand{\rarr}{\rightarrow}
\newcommand{\Rarr}{\Rightarrow}
\newcommand{\xrarr}[1]{\xrightarrow{#1}}
\newcommand{\larr}{\leftarrow}
\newcommand{\lrarr}{\leftrightarrows}
\newcommand{\rlarr}{\rightleftarrows}
\newcommand{\rrarr}{\rightrightarrows}
\newcommand{\al}{\alpha}
\newcommand{\bt}{\beta}
\newcommand{\ld}{\lambda}
\newcommand{\om}{\omega}
\newcommand{\Kd}[1]{\ensuremath{\Omega^{#1}}}
\newcommand{\KKd}{\ensuremath{\Omega^2}}
\newcommand{\vd}{\partial}
\newcommand{\PC}{\ensuremath{\mathbb{P}_1(\mathbb{C})}}
\newcommand{\PPC}{\ensuremath{\mathbb{P}_2(\mathbb{C})}}
\newcommand{\derz}{\ensuremath{\frac{\partial}{\partial z}}}
\newcommand{\derw}{\ensuremath{\frac{\partial}{\partial w}}}
\newcommand{\mb}[1]{\ensuremath{\mathbb{#1}}}
\newcommand{\mf}[1]{\ensuremath{\mathfrak{#1}}}
\newcommand{\mc}[1]{\ensuremath{\mathcal{#1}}}
\newcommand{\id}{\ensuremath{\mbox{id}}}
\newcommand{\dd}{\ensuremath{\delta}}
\newcommand{\bu}{\bullet}
\newcommand{\ot}{\otimes}
\newcommand{\boxt}{\boxtimes}
\newcommand{\op}{\oplus}
\newcommand{\mt}{\times}
\newcommand{\Gm}{\mathbb{G}_m}
\newcommand{\Ext}{\ensuremath{\mathrm{Ext}}}
\newcommand{\Tor}{\ensuremath{\mathrm{Tor}}}

\newcommand{\kn}[1]{\mathrm K(n)^*(#1)}
\newcommand{\ckn}[1]{\mathrm{CK}(n)^*(#1)}
\newcommand{\grckn}[1]{\mathrm{gr}_\tau^{*}\,\mathrm{CK}(n)^{*}(#1)}
\newcommand{\so}{\mathrm{SO}_m}
\newcommand{\pt}{\mathrm{pt}}
\newcommand{\tr}{\mathrm{tr}}
\newcommand{\sic}{\mathrm{sc}}
\newcommand{\ad}{\mathrm{ad}}
\newcommand{\sr}{\mathrm{sr}}
\newcommand{\St}{\mathrm{St}}
\newcommand{\SK}{\mathrm{SK}}
\newcommand{\SL}{\mathrm{SL}}
\newcommand{\Cp}{\mathrm{Cp}}
\newcommand{\Sp}{\mathrm{Sp}}
\newcommand{\Ep}{\mathrm{Ep}}
\newcommand{\Spin}{\mathrm{Spin}}

\newcommand{\A}{\mathsf{A}}
\newcommand{\C}{\mathsf{C}}

\def\GF#1{{\mathbb F}_{\!#1}}

\makeatletter
\newcommand{\colim@}[2]{%
  \vtop{\m@th\ialign{##\cr
    \hfil$#1\operator@font colim$\hfil\cr
    \noalign{\nointerlineskip\kern1.5\ex@}#2\cr
    \noalign{\nointerlineskip\kern-\ex@}\cr}}%
}
\newcommand{\colim}{%
  \mathop{\mathpalette\colim@{\rightarrowfill@\textstyle}}\nmlimits@
}
\makeatother

%%\newcounter{prob}
\newtheorem{lm}{Lemma}[section]
\newtheorem{lm*}{Lemma}
\newtheorem*{tm*}{Theorem}
\newtheorem*{tms*}{Satz}
\newtheorem{tm}[lm]{Theorem}
\newtheorem{prop}[lm]{Proposition}
\newtheorem*{prop*}{Proposition}
\newtheorem{prob}{Problem}%%%[prob]
\newtheorem{cl}[lm]{Corollary}
\newtheorem*{cor*}{Corollary}
\newtheorem{conj}{Conjecture}
\theoremstyle{remark}
\newtheorem*{rk*}{Remark}
\newtheorem*{rm*}{Remark}
\newtheorem{rk}[lm]{Remark}
\newtheorem*{xm}{Example}
\theoremstyle{definition}
\newtheorem{df}{Definition}
\newtheorem*{nt}{Notation}
\newtheorem{Def}[lm]{Definition}%[section]
\newtheorem*{Def-intro}{Definition}
\newtheorem{Rk}[lm]{Remark}
\newtheorem{Ex}[lm]{Example}

\theoremstyle{plain}
\newtheorem{Th}[lm]{Theorem}
\newtheorem*{Th*}{Theorem}
\newtheorem*{Th-intro}{Theorem}
\newtheorem{Prop}[lm]{Proposition}
\newtheorem*{Prop*}{Proposition}
\newtheorem{Cr}[lm]{Corollary}
\newtheorem{Lm}[lm]{Lemma}
\newtheorem*{Conj}{Conjecture}
\newtheorem*{BigTh}{Classification of Operations Theorem  (COT)}
\newtheorem*{BigTh-add}{Algebraic Classification of Additive Operations Theorem  (CAOT)}

\newtheorem{maintheorem}{Theorem}
\renewcommand{\themaintheorem}{\Alph{maintheorem}}

\newcommand\rA{\mathsf A}
\newcommand\rB{\mathsf B}
\newcommand\rC{\mathsf C}
\newcommand\rD{\mathsf D}
\newcommand\rE{\mathsf E}
\newcommand\rF{\mathsf F}
\newcommand\rG{\mathsf G}

\newcommand{\ep}{\epsilon}
\newcommand{\be}{\beta}
\newcommand{\ga}{\gamma}
\newcommand{\de}{\delta}
\newcommand{\la}{\lambda}
\newcommand{\vp}{\varphi}
\newcommand{\st}{\sigma}
\newcommand{\eps}{\varepsilon}

\newcommand{\lra}{\longrightarrow}
\newcommand{\idd}{\mathop{\rm{id}}\nolimits}
\newcommand{\Real}{{\mathbb R}}
\newcommand{\Co}{{\mathbb C}}
\newcommand{\komp}{{\mathbb C}}
\newcommand{\Int}{{\mathbb Z}}
\newcommand{\Nat}{{\mathbb N}}
\newcommand{\Rat}{{\mathbb Q}}

\def\mq{{\mathfrak q}}

\tikzcdset{
arrow style=tikz,
diagrams={>={Straight Barb[scale=0.8]}}
}

%%%%%%%%%%%%%%%%%%%%%%%%%

\title[Uniform bounded generation of Chevalley groups]{
%%Towards
Uniform bounded elementary generation\\
of Chevalley groups
%%over Dedekind  rings of arithmetic type
}

\author{Boris Kunyavski\u\i }
\address{%
Dept. of Mathematics \\
Bar-Ilan University \\
Ramat Gan, Israel
}
\email{kunyav@macs.biu.ac.il}

\author{Eugene Plotkin}
\address{%
Dept. of Mathematics \\
Bar-Ilan University \\
Ramat Gan, Israel
}
 \email{plotkin@macs.biu.ac.il}

%----------Author2
\author{\fbox{Nikolai Vavilov}}
\address{%
Dept. of Mathematics and Computer Science \\
St Petersburg State University \\
St Petersburg, Russia
}
\email{nikolai-vavilov@yandex.ru}

%% \author{}

\subjclass{20G07}

\keywords{Chevalley groups; Dedekind rings; bounded generation}

\thanks{Research of Boris Kunyavski\u\i{} and Eugene Plotkin was supported by the ISF grant 1994/20.
Nikolai Vavilov was supported by the ``Basis'' Foundation grant
N.\,20-7-1-27-1 ``Higher symbols in algebraic K-theory''. A part of this research was accomplished when Boris Kunyavski\u\i \ was visiting the IHES (Bures-sur-Yvette). Support of these institutions is gratefully acknowledged.}

\begin{abstract}
In this paper we establish a definitive result which almost completely closes the problem of bounded elementary generation for Chevalley groups of rank $\ge 2$ over arbitrary Dedekind rings $R$ of arithmetic type, with
uniform bounds. Namely, we show that for every reduced irreducible
root system $\Phi$ of rank $\ge 2$ there exists a
universal bound $L=L(\Phi)$ such that the simply connected Chevalley groups $G(\Phi,R)$ have elementary width $\le L$ for all Dedekind rings of arithmetic type $R$. %For %symplectic groups this result is new even in the number case. %The most difficult case
%that requires the most work is $\Sp(4,R)$.
\end{abstract}

\maketitle

\long\def\symbolfootnote[#1]#2{\begingroup%
\def\thefootnote{\fnsymbol{footnote}}\footnote[#1]{#2}\endgroup}

\symbolfootnote[0]{Nikolai Vavilov suddenly passed away on the 14th
of September, 2023. It was him who initiated this research. Everybody
who had a privilege to work with Nikolai and who had been touched by
his mathematical vision and intuition, comprehensive and profound erudition,
and generous personality will never forget him. He will be deeply missed.
Let the memory of our friend be blessed. {\it
{B.K., E.P.}}}

\section*{Introduction and State of Art}

In the present paper, we consider Chevalley groups
$G=G(\Phi,R)$ and their elementary subgroups
$E(\Phi,R)$  over Dedekind rings
of arithmetic type.  Usually it is more convenient
to speak of the simply connected group
$G_{\sic}(\Phi,R)$.  In most of the cases we are interested in it coincides with the elementary group
$E_{\sic}(\Phi,R)$. When there is no danger of confusion, we drop any indication of the weight lattice.

Our ring $R$ is an arbitrary  Dedekind ring
of arithmetic type, which means that throughout the paper one has to distinguish the corresponding  number and function cases.

\par
%Primarily, here
We occupy ourselves with the classical problem of estimating
the width of $E(\Phi,R)$ with respect to the elementary generators $x_{\alpha}(\xi)$,
$\alpha\in\Phi$, $\xi\in R$. %and also with the similar
%problem of estimating the width with respect to
%commutators.
We consider the subset $E^L(\Phi,R)$
consisting of products of $\le L$ such elementary generators.
The {\bf elementary width} is defined as the smallest
$L$ such that each element of $E(\Phi,R)$ can be represented as a product of $\le L$ elementary
generators $x_{\alpha}(\xi)$, in other words,
$$ E(\Phi,R)=E^L(\Phi,R). $$
\noindent
If there is no such $L$,
we say that the width is infinite. If the width is finite,
we say that $G$ is {\bf boundedly elementarily generated}.
%if $E(\Phi,R)$ has finite
%width $w_E(G)$ with respect to elementary generators.
\par
Let us start with a short review of early works on the topic. Most of them, and many papers even today
only treat the special case of $\SL(n,R)$.
The pioneering 1975 paper by George Cooke and Peter Weinberger \cite{CW} showed that, with the %possible
exception of $\SL(2,R)$ over very meagre rings,
such as $R=\Int, \GF{q}[t]$, or other arithmetic rings
with the finite multiplicative group, the problem of
bounded elementary generation admits a positive {\bf uniform} solution.
%and at that a one.
In other words, in this case
there exists a bound $L=L(\Phi)$ depending solely on $\Phi$
%alone
 such that the elementary width of $G(\Phi,R)$
over all Dedekind rings $R$ of arithmetic type does not exceed $L$.
\par
However, their actual proofs were {\it conditional\/},
they depended on a very strong form of the
{\bf GRH} = Generalised Riemann Hypothesis. The
most important early contributions towards obtaining
{\it unconditional\/} proofs of such results over
{\it number rings\/}, are due
to David Carter and Gordon Keller, 1983--1985.
\par\smallskip

%$\bullet$ The arithmetic proofs for  with {\it explicit\/}
%bounds depending not only on $\Phi$, but on $R$,
%or, at least on the degree of the extension $K/\Rat$,
%where $K$ is the fraction field of $R$,
%and some further arithmetic invariants of $R$, such
%as, for instance, [the number of prime divisors of] the discriminant of $R$, see \cite{CaKe, CaKe2}.
%\par\smallskip

$\bullet$ The arithmetic proofs  for $\SL(n,R)$, $n\geq 3$, with {\it explicit\/}
bounds depending not only on $\Phi$, but also on some arithmetic invariants of $R$ were obtained in \cite{CaKe, CaKe2},
%or, at least on the degree of the extension $K/\Rat$,
%where $K$ is the fraction field of $R$,
%and some of $R$, such
%as, for instance, [the number of prime divisors of] the discriminant of $R$, see \cite{CaKe, CaKe2}.
\par\smallskip

$\bullet$ For the model theoretic proofs in the number case, which yield %prove
the {\it existence\/} of bounds $L=L(\Phi,d)$ depending
on $\Phi$ and the degree $d=|K:\Rat|$, non-constructive, without presenting any actual %such
bounds, see for instance the truly remarkable [but unfortunately still unpublished] preprint by Carter and Keller with Eugene Paige \cite{CKP}, and its re-exposition by Dave Morris
\cite{Mor}.
\par\smallskip
$\bullet$ Around 1990 Oleg Tavgen \cite{Tavgen1,
Tavgen, Tavgen2} succeeded in
generalising these results to all Chevalley groups
of normal types, and to most twisted Chevalley groups.
With this end, he invented a very slick reduction trick,
which reduced the study of bounded generation to
rank 2 cases, essentially to $\SL(3,R)$ and $\Sp(4,R)$,
and was able to solve the cases of $\Sp(4,R)$ and
$G(\rG_2,R)$ by direct matrix computations imitating
the arithmetic proof by Carter and Keller.
%So, by works of Tavgen the bounds of the type $L=L(\Phi,d)$ were obtained for all Chevalley groups $G=G(\Phi,R)$ of rank bigger than 1.
As the Carter--Keller bounds, Tavgen's ones depended on arithmetic invariants of $R$.
\par\smallskip
$\bullet$ The only published result for the function case
until rather recently was the very early 1975 paper by
Clifford Queen \cite{Queen}, who established {\it the best
possible\/} absolute bound $L=5$, but only for {\it some\/}
function rings with infinite multiplicative group subject
to further arithmetic conditions. Even the case of
$R=\GF{q}[t]$ remained open at that stage.
\par\smallskip
Such was the state of art around 1990, and the
results listed above remained almost unrivaled for about two more decades.
There were some interesting attempts to come up
with explicit bounds (compare, for instance \cite{Liehl2,
LoMu, Murty}), but the resulting bounds always depended on some further arithmetic invariants and/or worked only under some severe restrictions on $R$.
\par
%However, there were many connections with
%other related problems --- add several sentences here!!!
%--- and with verbal width etc.

However, there were many reasons which eventually led to
a new surge of activity in this direction starting
around 2010\footnote{We ourselves learned about
the status of this problem as unsolved from Sury,
see \cite{VSS}. In particular, it is proved in \cite{VSS}
that the only rings for which one has $G=UU^-UU^-$
are rings with stable range 1, %$\sr(\xi)=1$,
so that $L=5$ is the best possible bound for $\SL_2$ over arithmetic rings %for which $\sr(\xi)=1{1\over2}$.}.
with stable range $1\frac12$.}.
Let us mention   relations with the congruence subgroup property, Kazhdan property T, Waring-type problems for groups, model theoretic applications, and so on.

An important initial breakthrough,
the first {\it unconditional\/} proof of the bounded
generation of $\SL\Big(2,\Int\Big[\displaystyle{\frac1p}\Big]\Big)$
with an {\it explicit\/} bound, and at that the best possible
one, $L=5$, was achieved by Maxim Vsemirnov \cite{Vs}.
\par
Let us list the key contributions of the last 5 years which
together essentially amount to the complete solution of
the problem.
\par\smallskip
$\bullet$ For the {\it number case\/}, when $R^*$
is infinite there is a definitive result for $\SL(2,R)$ by Morgan, Rapinchuk and Sury~\cite{MRS} in 2018,
with a small uniform bound $L\le 9$, which can be improved [and was improved!] in some cases.
Thus, Bruce Jordan and Yevgeny Zaytman \cite{JZ}
improved it to $L\le 8$ (and further improved to $L\le 7$
or $L\le 6$ in the presence of finite or real valuations in $S$).
\par\smallskip
$\bullet$ In the same 2018 Bogdan Nica \cite{Nica}
has finally established bounded elementary generation
of $\SL(n,\GF{q}[t])$, $n\ge 3$. He proposed a slight
variation of Carter--Keller's approach, replacing the
full multiplicativity of Mennicke symbol by a weaker
form, ``swindling lemma''. This is where we jumped in.
In \cite{KPV} we developed reductions of all
non-symplectic Chevalley groups to $\SL(3,\GF{q}[t])$, and
devised a similar proof for $\Sp(4,\GF{q}[t])$.
\par\smallskip
$\bullet$ The decisive contributions in the function case
are due to Alexander Trost \cite{trost,trost2}, who succeeded in proving
versions of all necessary arithmetic lemmas
in the function case. Actually, his versions are
{\it better\/} than the corresponding results in the
number case\footnote{One of the reasons is that
adjoining roots of unity in the number case one gets
a cyclotomic extension which may have non-trivial ramification, whereas in the function case one gets a constant extension, which is not ramified.}. In particular, he gave an {\it explicit uniform\/} bound for the bounded elementary
generation of $\SL(n,R)$, $n\ge 3$, which does not
depend on the degree $d=[K:\GF{q}(t)]$.
\par\smallskip
$\bullet$ Finally, the recent paper by Kunyavskii, Morris and Rapinchuk \cite{KMR} improves the uniform bound
for $\SL(2,R)$, for rings $R$ with infinite multiplicative group $R^*$ to $L\le 7$ in the number case and establishes a
similar result with the bound $L\le 8$ in the
{\it function case\/}.
\par\smallskip
Thus, the results of~\cite{CKP, Mor, MRS, trost2, KMR}
{\it completely\/}
solve the problem of the uniform bounded elementary generation
for the special linear groups $\SL(n,R)$, $n\ge 3$, ---
and when $R^*$ is infinite, even for $\SL(2,R)$.
\par
%An unexpected observation is that
The methods of our previous paper~\cite{KPV} completely reduce the proof
of a similar result for almost all other Chevalley groups, including even the
{\it symplectic\/} groups $\Sp(2l,R)$, $l\ge 3$, to the case of $\Phi=\A_2$. %Thus,
{\it The only\/} case
that does not follow rightaway by combining results
of the above papers, is that of $\Sp(4,R)$.
\par
Here we solve the remaining case of $\Sp(4,R)$ and
%% Combining the methods of~\cite{KPV and
%% \cite{trost2}, we are now able to
thus come up with a complete solution of uniform bounded generation for Chevalley groups in
the general case.

\begin{maintheorem}\label{th_A}
Let\/ $\Phi$ be a reduced irreducible root system of rank
$l\ge 2$. Then there exists a constant $L=L(\Phi)$, depending on $\Phi$
alone, such that for any Dedekind ring of arithmetic type $R$,
any element in $G_{\sic}(\Phi,R)$ is a product of at most $L$ elementary root unipotents,
$$ G_{\sic}(\Phi,\,R)=E^L(\Phi,R). $$
\end{maintheorem}

\begin{rk}
The bounds obtained in Theorem A are {\it uniform} with respect to $R$, both in
the number and function cases. What is important --- and unexpected! ---
in the function case they are {\it explicit}.
\end{rk}

\begin{rk}
In the number case explicit bounds are only available when $R^*$ is infinite.
%the statement of Theorem A is new for symplectic groups.
\end{rk}

%An important --- and unexpected! ---
%aspect of this work is the existence of {\it explicit uniform\/} bounds in the {\it function} case. In the number
%case the bounds are also uniform, but if we wish to cover
%all $R$ and not just those with infinite multiplicative
%group $R^*$ they are not explicit. Note that for symplectic groups Theorem \ref{th_A} is new even
%in the number case.
%\end{rk}

\begin{rk}
Theorem A was already announced in \cite{KLPV},
with a sketch of proof. However, since \cite{KLPV} is focused on bounded generation for the Steinberg groups,
it would be unreasonable to provide there tedious computational aspects of the proof for Chevalley groups.
Therefore, the most tricky case of $\Sp(4,R)$ was skipped there, in several cases not needed for the treatment
of Steinberg groups the arguments were only briefly sketched, and no care of explicit numerical bounds was taken.
Here, we supply
all the details for the case of $\Sp (4,R)$. Moreover, we redo
the case $\SL(3,R)$ for the function rings, which was already solved in \cite{trost2}. However, we do it in
the style of \cite{Nica}, rather than \cite{CaKe},
which allows us to improve the estimate for $L$
from $L\le 65$ to $L\le 44$. This improvement
then gives slightly better bounds in all explicit
estimates for all other Chevalley groups in the function case.
\end{rk}

\begin{rk}
Note that uniform estimates, being interesting in their own right, are indispensable
for some applications, e.g. for estimating Kazhdan constants of arithmetic groups, see
\cite{Had}.
\end{rk}

Roughly, the ingredients of the proof are as follows.

\par\smallskip
$\bullet$ We consider the case where $R^*$ is infinite separately and prove the following statement.

\begin{maintheorem}\label{th_B}%%% \phantom{p}\quad
\label{infinite}
For any Dedekind ring of arithmetic type $R$ with
the infinite multiplicative group $R^*$ any element in
$G_{\mathrm{sc}}(\Phi,R)$ is a product of at most $L=7N$ elementary unipotents in the number
case or $L=8N$ elementary unipotents in the
function case, where $N=|\Phi^+|$ is the number of positive roots of $\Phi$.
\end{maintheorem}

The proof of Theorem \ref{th_B} is cheap modulo deep results for rank 1 case and requires Tavgen's reduction trick.
This is done in Section \ref{sec:Tavgen}. Thus in the proof of Theorem \ref{th_A} we may assume that $R^*$ is finite. This is important in the
number case.

%%% Interestingly,
\par\smallskip
$\bullet$ In the {\bf function} case our proofs in this paper have %[next to]
almost zero
arithmetic component. Namely, all
arithmetic results we need are taken essentially
{\sc as is} from the paper by %Alexander
Trost \cite[Lemma 3.1 and Lemma 3.3]{trost2}.
After that %in the function case
the rest of the proof is a pure theory of
algebraic groups and some stability theorems from algebraic K-theory.
\par
More precisely, we show --- this part is indeed essentially
contained already in \cite{KPV} --- that for all
non-symplectic Chevalley groups bounded generation is reduced to that for $\SL(3,R)$.
What has been overlooked in~\cite{KPV} though,
is that bounded generation of $\Sp(2l,R)$, $l\ge 3$, also reduces to $\SL(3,R)$, with the help of the
symplectic lemmas on switching long and short roots \cite{KPV}. Only after rediscovering this trick ourselves in March 2023 we noticed that a similar approach has been
used by Kairat Zakiryanov \cite{Zak}, and this reference
should have been included in
\cite{KPV}. %\footnote{There are reasons that explain why Zakiryanov's contribution was largely ignored.
%He wrongly claimed that the group $\Sp(4,\Int)$ is
%not boundedly generated. He considered bounded generation in terms of elementary generators {\it and\/}
%some semisimple elements, etc. Nevertheless!}.
\par

For the only remaining case $\Sp(4,R)$ we can also obtain
an {\it explicit uniform\/} bound by combining the
arithmetic lemmas of Trost \cite{trost2} with our $\Sp_4$-lemmas from \cite{KPV}, in {\it exactly\/} the same style as in \cite[Section~5]{KPV}, and that is by far the most difficult part of the proof.
Namely, Trost's Lemma~3.1 is essentially
a generalisation of our Lemma~5.4, we only have to
supplement it slightly in characteristic 2. Trost's Lemma~3.3
shows that --- unlike for the number fields! ---
the case of a general function ring $R$
is not much different from the case when $R$ is a PID
(and, in particular, there is no dependence on degree
or other invariants).
%%% The key ingredient for this, bounded extraction
%%% of square roots from Mennicke symbols, is also
%%% contained in \cite{trost2}.

Note that in a more recent preprint, Trost
established the uniformly bounded generation in the symplectic case
(with weaker estimates), see \cite[Corollary~3.11]{trost3}.

\par\smallskip
$\bullet$ The {\bf number} case is very different.
Our general strategy is similar to the proof of \cite[Theorem~4.1]{trost2}.
Namely, we use
very deep arithmetic results of \cite{MRS}
(or any of their improvements in \cite{JZ, KMR})
pertaining to $\SL(2,R)$ to prove Theorem B asserting that there is
an explicit uniform bound when $R^*$ is infinite.
{\it Some\/} such bounds can be easily derived by
a version of the Tavgen's trick~\cite[Theorem~1]{Tavgen}, as described and generalised
in~\cite{VSS, SSV} and \cite{KPV}.
\par
%%% After that it's some model theory.
We are left with
the rings of integers of the imaginary quadratic
fields $K$, and thus with a single degree $d=2$.
Since this class is contained in a class defined by
the first order conditions and sharing uniform
estimates of the congruence kernel,
using non-standard models (alias ultrafilters,
alias compactness theorem in the first order logic,
alias\ldots) one can then prove the following:
if all $\SL(3,R)$ are boundedly elementarily generated,
they are {\it uniformly\/} boundedly elementarily generated. This argument was devised by
Carter--Keller--Paige \cite{CKP}, and then rephrased
slightly differently by Morris \cite{Mor} (see also the
discussion in \cite{trost, trost2}).
%%% The uniform bound for $\SL(3,R)$ over imaginary
%%% quadratic rings was obtained by
\par
Since all other cases, except  $\Sp(4,R)$, are reduced
to $\SL(3,R)$ by the standard tricks collected in
\cite{KPV}, we are again left with $\Sp(4,R)$ alone.
Of course, in the {\it number case\/} the bound given
for $\Sp(4,R)$ by Tavgen \cite{Tavgen} is not uniform,
it depends on the degree {\it and\/} the discriminant
of the number field $K$. However, since $\Sp(4,R)$
and its elementary generators are described by first
order relations, we can again use exactly the same argument of \cite{CKP, Mor} to conclude that there
exists an absolute constant as an upper bound for the width of all $\Sp(4,R)$, where $R$ is the ring of
integers of an imaginary quadratic number field.
Of course, now we know only that {\it some\/} such constant exists, it is by no means explicit.
\par\smallskip

The paper is organised as follows.
%%%%%%%%%%%%Izmenit` vse! %%%%%%%%%%%%%%%%%%%%%%%%%%
In Section~\ref{prelim}
we recall notation and collect some
preliminary results.
%%% which constitute a part of the proofs of
%%% Theorems A, B and C,
%%% adding more details whenever needed.
In Section~\ref{support} we recall some important arithmetic lemmas.
In Section~\ref{sec:Tavgen} we prove Theorem B and also the cases of Theorem A corresponding to the simply laced root systems $\Phi$ and $\Phi=\mathsf F_4$
(we collect these cases in Theorem C).
Section~\ref{stability} deals with surjective stability of K$_1$-functor and consequences.
In Section~\ref{sec:sw} we prove a swindling lemma for the groups $\SL(3,R)$ and $\Sp(4,R)$.
In Section~\ref{sec:sl3f} and \ref{sec:sp4f} we establish new bounds for the width of $\SL(3,R)$ and $\Sp(4,R)$ in the function case.
Section~\ref{sec:sp4n} is devoted to $\Sp(4,R)$ in the number case. Finally, Section \ref{sec-conj}  contains some concluding remarks and open
problems.

%%%%%%%%%%%%%%%%%%%%%%%%%%%%%%%%%%%%%%%%%%%%%%%%%%%%%%%%%%%%%%%%%%%%%%%%%%%%%%%

\section{Notation and preliminaries}  \label{prelim}

In this section we briefly recall the notation, mainly taken from \cite{KLPV},
that will be used throughout the paper and some background
facts. For more details on Chevalley groups over rings
see \cite{Vav91} or \cite{VP}, where one can
find many further references.

%For a group $G$ and a subset $X\subseteq G$ such that $X^{-1}=X$, we say that $G$ is boundedly generated by $X$ if there exists a constant $C$ such that any element $g\in G$ is a product of at most $C$ elements of $X$.

\subsection{Chevalley groups}
\def\wP{\mathcal P}
\def\wQ{\mathcal Q}
Given a reduced root system $\Phi$ %of rank $\ge 2$ and $W=W(\Phi)$ be its Weyl group. In our main
(usually assumed irreducible), %though
%in some proofs one has to use subsystems that are
%not. As usual, we choose an order on $\Phi$
we denote by $\Phi^+$, $\Phi^{-}$ and
$\Pi=\big\{\alpha_1,\ldots,\alpha_l\big\}$ the
sets of positive, negative and fundamental roots with respect to
a chosen order. Throughout we denote $N=|\Phi^+|$.

For a lattice $\wP$ intermediate
between the root lattice $\wQ(\Phi)$ and the weight
lattice $\wP(\Phi)$ and any commutative unital ring $R$
with the multiplicative group $R^*$, we denote by  $G=G_{\wP}(\Phi,R)$
the Chevalley group
of type $(\Phi,\wP)$ over $R$. %It is usually constructed as the
%group of $R$-points of the Chevalley--Demazure
%group scheme $G_{\wP}(\Phi,\text{$-$})$ of type $(\Phi,\wP)$.
In the case
$\wP=\wP(\Phi)$ the group $G$ is called simply connected and
is denoted by $G_{\sic}(\Phi,R)$. In another extreme case
$\wP=\wQ(\Phi)$ the group $G$ is called adjoint and
is denoted by $G_{\ad}(\Phi,R)$.
\par
Many results do not depend
on the lattice $\wP$ and
%hold for all groups of a given
%type $\Phi$.
%In all such cases, or when $\wP$ is determined by
%the context,
we often omit any reference to $\wP$ in the notation
and denote by $G(\Phi,R)$ {\it any} Chevalley group of type
$\Phi$ over $R$.
%However in some cases specific bounds
%may depend on $\wP$.
Usually, by default we assume
that $G(\Phi,R)$ is simply connected, but in some cases
it is convenient to work with the adjoint group, which
is then reflected in the notation.
\par
Fixing a split maximal torus $T=T(\Phi,R)$ in $G=G(\Phi,R)$ and identifying $\Phi$ with $\Phi(G,T)$,  %This choice
%uniquely determines
we denote by $X_{\alpha}$,
$\alpha\in\Phi$, the unipotent root subgroups in $G$, elementary with respect to $T$. %As usual,
We fix maps $x_{\alpha}\colon R\mapsto X_{\alpha}$, so that
$X_{\alpha}=\{x_{\alpha}(\xi)\mid \xi\in R\}$, and require that these parametrisations are interrelated by the Chevalley commutator formula with integer coefficients, see \cite{Carter},
\cite{Steinberg}. The above unipotent elements
$x_{\alpha}(\xi)$, where $\alpha\in\Phi$, $\xi\in R$,
elementary with respect to $T(\Phi,R)$, are also called
[elementary] unipotent root elements or, for short, simply
root unipotents.
\par
Further,
$$ E(\Phi,R)=\big\langle x_\alpha(\xi),\ \alpha\in\Phi,\ \xi\in R\big\rangle $$
\noindent
denotes the {\it absolute\/} elementary subgroup of $G(\Phi,R)$,
spanned by all elementary root unipotents, or, what is the
same, by all [elementary] root subgroups $X_{\alpha}$,
$\alpha\in\Phi$.
For $\epsilon\in\{+,-\}$ denote
$$
U^\epsilon(\Phi,\,R)=\left\langle x_\alpha(\xi)\mid\alpha\in\Phi^\epsilon,\ \xi\in R\right\rangle\leq\mathrm{E_{sc}}(\Phi,\,R).
$$

\subsection{Relative groups}
Let $\mq\unlhd R$ be an ideal of $R$, and let
$\rho_\mq\colon R\longrightarrow R/\mq$ be the
reduction modulo $\mq$. By functoriality, it defines the reduction homomorphism of Chevalley groups
$\rho_\mq\colon G(\Phi,R)\longrightarrow G(\Phi,R/\mq)$. The kernel
of $\rho_\mq$ is denoted by $G(\Phi,R,\mq)$ and is called the principal congruence subgroup of $G(\Phi,R)$ of level $\mq$.
We denote by
$X_{\alpha}(\mq)$ the intersection of $X_{\alpha}$ with the principal congruence
subgroup $G(\Phi,R,\mq)$. Clearly, $X_{\alpha}(\mq)$ consists of all elementary
root elements $x_{\alpha}(\xi)$, $\alpha\in\Phi$, $\xi\in \mq$, of level $\mq$:
$$ X_{\alpha}(\mq)=\big\{x_{\alpha}(\xi)\mid\xi\in \mq\big\}. $$
\noindent
By definition, $E(\Phi,\mq)$ is generated by $X_{\alpha}(\mq)$, for all roots
$\alpha\in\Phi$. The same subgroups generate $E(\Phi,R,\mq)$ as a {\it normal\/}
subgroup of the absolute elementary group $E(\Phi,R)$.
\par
The classical Suslin--Kopeiko--Taddei theorem asserts that for $\mathrm{rk}(\Phi)\ge 2$ one has
$E(\Phi,R,\mq)\unlhd G(\Phi,R)$. The quotient
$${\mathrm K}_1(\Phi,R,\mq)=G_{\sic}(\Phi,R,\mq)/E_{\sic}(\Phi,R,\mq) $$
\noindent
is called the [relative] K$_1$-functor. The absolute case
corresponds to $\mq=R$,
$${\mathrm  K}_1(\Phi,R)=G_{\sic}(\Phi,R)/E_{\sic}(\Phi,R). $$
\noindent
Observe
$${\mathrm  K}_1(\rA_l,R,\mq)=\SK_1(l+1,R,\mq), $$
\noindent
so that our K$_1$-functor corresponds rather to
the $\SK_1$ of the classical theory.

%%% Generators of $E(\Phi,R,I)$ {\it as a group\/}
%%% are recalled in the next section.

\subsection{Arithmetic case}

For a global field $K$ and a finite non-empty set $S$ of places of $K$ (containing all archimedean places when $K$ is a number field),
let
$$
R=\{x\in K\mid v(x)\geq0\ \forall v\not\in S\}.
$$
\noindent
It is a Dedekind domain whose maximal ideals can be canonically identified with with the places outside $S$.
Following \cite{BMS}, we say that $R$ is {\it the Dedekind ring of arithmetic type} defined by the set $S$
(or, for short, an arithmetic ring).
\par
For the arithmetic rings Bass, Milnor and Serre \cite{BMS} have explicitly calculated $K_1(\Phi,R,\mq)$,
$\Phi=\rA_l,\rC_l$, $l\ge 2$, in terms
of Mennicke symbols. Namely, they have proven that
${\mathrm K}_1(\rA_l,R,\mq)\cong C(\mq)$ and
${\mathrm K}_1(\rC_l,R,\mq)\cong\Cp(\mq)$
(the universal Mennicke groups), which in turn are then
identified via reciprocity laws with certain groups of
roots of 1 in $R$.
\par
The [almost] positive solution of the congruence
subgroup problem for these groups amounts to the
fact that the {\bf congruence kernel}
$$ C(G):=\lim_{\longleftarrow} C(\mq), $$
\noindent
taken over all non-zero ideals $\mq\unlhd R$ is
finite. Actually, it is trivial, apart from the case when
$R$ is the ring of integers $\mathcal O_K$ in a purely
imaginary number field $K$, when $C(G)\cong\mu(K)$
is the groups of all roots of 1 in $K$.
\par
Later their results were generalised to all Chevalley
groups by Hideya Matsumoto \cite{mats}. The following
special case of his results
\cite[Th\'eor\`eme~12.7]{mats} explains why
we usually prefer to work with simply connected
groups.

\begin{lm}%% [Matsumoto]
\label{e=g}
Let $R$ be a Dedekind ring of arithmetic type and
$\Phi$ a reduced irreducible root system of rank at
least $2$. Then
$$
\mathrm E_{\mathrm{sc}}(\Phi,\,R)=G_{\mathrm{sc}}(\Phi,\,R).
$$
\end{lm}
%%% In fact, for $\Phi=\mathsf A_l,\,\mathsf C_l$
%%% was established already by Hyman Bass,
%%% John Milnor and Jean-Pierre Serre in~\cite{BMS}.
%%% Recently Anastasia Stavrova generalised it
%%% to isotropic reductive groups and to polynomial
%%% rings over $R$, see~\cite[Corollary~1.2]{stavrova}.

%%%%%%%%%%%%%%%%%%%%%%%%%%%%%%%%%%%%%%%%%%%%%%%%%%%%%%%%%%%%%%%%%%%%%%%%%%%%%%%%%%%%%%%%%

\section{Supporting statements} \label{support}

\subsection{Reduction to the ring of integers}
The following result is a combination of
\cite[Lemma 2.1]{BMS} and \cite[Lemma 5.3]{BMS}.
The same proof, with several successful deteriorations,
is reproduced on page 685 of \cite{CaKe}.
%%% and then parts of it, with more details but further deteriorations

\begin{lm}\label{local}
Let $R$ be a Dedekind ring, $s\in R$, $s\neq 0$.
Then
$$ \SL\Big(2,R\Big[{\frac 1s}\Big]\Big)=
\SL(2,R)E^3\Big(2,R\Big[{\frac 1s}\Big]\Big). $$
\end{lm}
In other words, every $2\times 2$ matrix with
entries in $R\big[{\frac 1s}\big]$ can be
reduced to a matrix with entries in $R$ by
$\le 3$ elementary moves with parameters in
$R\big[{\frac 1s}\big]$. Since the number of
elementary moves during the rank reduction does
not depend on the ring $R$, and the only
such dependence occurs at the base of induction,
we immediately get the following corollary.

\begin{lm}\label{localisa}
Let $R$ be a Dedekind ring such that any element
of $\mathrm{G_{sc}}(\Phi,\,R)$ is a product of
$L$ elementary root unipotents. Then for any
$s\in R$, $s\neq 0$, any element of
$\mathrm{G_{sc}}(\Phi,\,R\big[{\frac 1s}\big])$ is a product of at most $L+3$ elementary root unipotents.
\end{lm}

\subsection{Arithmetic lemmas}
The following lemma is the arithmetic heart of the
whole proof. In the number case it is \cite[Lemma 1]{CaKe} and in the function case it was first
proven in full generality in \cite[Lemma 3.1]{trost2}
(before that only a special case $R=\GF{q}[t]$
was established as \cite[Lemma 6.4]{KPV}).

\begin{lm} \label{lem:root}
Let $\mathcal O_K$ be the ring of integers of a global field $K$, and let $x\in\SL(2,\mathcal O_K)$.
Let $m$ be the number of roots of $1$ in $K$
in the number case, respectively $m=q-1$, where
$\GF{q}$ is the field of constants of $K$ in the
function case. Then for any matrix $A\in\SL(2,\mathcal O_K)$
there exist nonzero elements $a,b\in\mathcal O_K$ with
the following properties:
\par\smallskip
$\bullet$ $b\mathcal O_K$ is a prime ideal and, moreover, in the
number case $b\mathcal O_K$ is unramified in $K/\mathbb Q$, and
does not contain $m$.
\par\smallskip
$\bullet$ $A$ can be transformed to a matrix with the
first row $(a^m,b)$ by means of not more than $4$ elementary  moves in the number case, or $3$ elementary moves in the function case.
\end{lm}

\begin{rk}
In the function case, we shall need a version of Lemma \ref{lem:root}
with $m=2$, to be able to extract square roots of Mennicke symbols.
If $q$ is odd, such a version follows automatically because we then have
$a^{q-1}=(a^{(q-1)/2)})^2.$ If $q$ is even, we apply the argument from the end of the
proof of \cite[Lemma~6.4]{KPV}.
\end{rk}

For a global function field $K$ with the field of constants $\GF{q}$ and
$b\in \mathcal O_K$, $b\neq 0$, we denote by
$\epsilon(b)$ the exponent of the [finite] multiplicative group $(\mathcal O_K/b\mathcal O_K)^*$ and set
$\delta(b)=\e(b)/(q-1)$. The following result is due to Trost \cite[Lemma~3.3]{trost2}.

\begin{lm} \label{lem:unram}
Let $K$ be a global function field with the field of constants $\GF{q}$, $a,b\in{\mathcal O}_K\setminus\{0\}$, such that $b{\mathcal O}_K$ is prime and $a$ and $b$ are comaximal, $a{\mathcal O}_K+b{\mathcal O}_K={\mathcal O}_K$. Then for every
unit $u\in {\mathcal O}_K^*$ there exists
$c\in {\mathcal O}_K$ such that
\par\smallskip
$\bullet$ $bc\equiv u\pmod a$,
\par\smallskip
$\bullet$ $\delta(b)$ and $\delta(c)$ are coprime.
\end{lm}

%%%%%%%%%%%%%%%%%%%%%%%%

\section{Tavgen rank reduction theorem and applications} \label{sec:Tavgen}

\subsection{Tavgen rank reduction theorem}
In this Section we prove Theorem \ref{th_B} and establish some other useful consequences of Tavgen's reduction theorem.

The following trick allowing one to reduce the rank of a root system under consideration %$\A_1$ or $\A_2$,
was invented by Tavgen \cite{Tavgen}
(and then generalised in \cite{VSS} and
\cite{SSV}). The following final form is proven in
{\cite[Theorem~3.2]{KPV}}.

\begin{lm}\label{tavgen}
Let $\Phi$ be a reduced irreducible root system of rank $l\geq2$, and $R$ be a commutative ring. Let $\Delta_1,\ldots,\Delta_t$ be some subsystems of $\Phi$, whose union contains all fundamental roots of $\Phi$. Suppose that for all $\Delta_i$ the elementary Chevalley group $\mathrm{E_{sc}}(\Delta_i,\,R)$ admits a unitriangular factorisation
$$
\mathrm{E_{sc}}(\Delta_i,\,R)=U^+(\Delta_i,\,R)\,U^-(\Delta_i,\,R)\,U^+(\Delta_i,\,R)\ldots U^{\pm}(\Delta_i,\,R)
$$
of length $N$ {\rm(}not depending on $i${\rm)}. Then the elementary group $\mathrm{E_{sc}}(\Phi,\,R)$ itself admits unitriangular factorisation
$$
\mathrm{E_{sc}}(\Phi,\,R)=U^+(\Phi,\,R)\,U^-(\Phi,\,R)\,U^+(\Phi,\,R)\ldots U^{\pm}(\Phi,\,R)
$$
of the same length $N$.
\end{lm}

It is used below in two cases, when
all $\Delta_i$'s are $\A_1$, and when all of them are
$\A_2$.

%%%%%%%%%%%%%%%%%%%%%%%%

\subsection{The case when $R^*$ is infinite}
The case where a Dedekind ring $R$ of arithmetic type has infinitely many units
is now {\it completely\/} solved, with very small {\it absolute} constant.
%We cannot describe the whole chain of events here, and mention all contributors.
Here is a brief account of main steps along this route.
Vsemirnov \cite{Vs} established a first unconditional
result of this sort, not depending on the GRH,
%Aleksander
Morgan, %Andrei
Rapinchuk and Sury
\cite{MRS} proved that $\SL(2,R)$ is boundedly elementarily generated
in number case for an arbitrary $R$ with infinite $R^{\ast}$. The absolute bound obtained in their paper is  $L=9$.

In the paper presently under way
the first author, Morris and Rapinchuk
\cite{KMR} improved the bound to $L=7$ in the number case (which we believe is the best possible and cannot
be further improved, in general). A
similar result holds in the function case, with the bound
$L=8$ (which, we believe, can be further improved to $L=7$). %Hence,

\begin{lm}{\cite{KMR}} %%\phantom{p}\quad
\label{sury}
For any Dedekind ring of arithmetic type $R$ with
the infinite multiplicative group $R^*$ any element in
$\SL(2,R)$ is a product of at most $7$ elementary transvections in the number case or at most $8$
elementary transvections in the function case.
\end{lm}

%So, Theorem B is getting an immediate consequence of the above mentioned results. Recall its formulation:

%{\it
%\begin{Th}%%% \phantom{p}\quad
%\label{th_B}
%For any Dedekind ring of arithmetic type $R$ with
%the infinite multiplicative group $R^*$ any element in
%$\mathrm G_{\mathrm{sc}}(\Phi,R)$ is a product of at most $L=7N$ elementary unipotents in the number
%case or $L=8N$ elementary unipotents in the
%function case. Here, as always, $N=|\Phi^+|$.
%$\end{Th}
%}

Together with Lemma \ref{tavgen}, this immediately implies Theorem B:
it covers the case $\Delta =\mathsf A_1$, and all higher rank cases
are reduced to rank one by putting $\Delta_i=\mathsf A_1$ in Lemma \ref{tavgen}.

%\begin{proof} The case $\Delta =A_1$ is covered by Lemma \ref{sury}. It remains to take $\Delta_i=A_1$ in Lemma \ref{tavgen}.
%\end{proof}

%Combining \ref{tavgen} and \ref{sury} we immediately
%get the following stronger form of Theorem~A
%in this special case, with {\it explicit\/} bounds.
%Here, as always, $N=|\Phi^+|$.

Thus the condition $|R^*|=\infty$ makes a huge relief.
%It is indeed shocking
%that one does not
Essentially no extra work is needed to treat the
general case with not the best possible but still rather
plausible bounds (anyway, asymptotically $L$ cannot be smaller
than something like $3N$ to $4N$).

%???? Kak mw poluxili. Xem pokryvali???

So, if we are not interested in actual bounds, but
just in uniform boundedness, the rest of the exposition
is formally dedicated to the Dedekind rings of arithmetic type
with {\it finite\/} multiplicative groups.
Thus in the number case, we restrict our attention  to the rings of
integers in imaginary quadratic number fields
(since $\mathbb Z$ is already covered). %. ??? a Z
%\par
%???? Neqsno, xto imennno sobiraemsja dokazwvat`
In the function case, as discovered by Trost \cite{trost2},
for ranks $\ge 2$ we
do not have to distinguish between rings with finite
and infinite multiplicative group, so that the rest of this
section does not depend on \cite{KMR} (but does
depend on \cite{MRS}).

\subsection{The simply laced case and $\Phi=\mathsf F_4$}
In this Section we prove the statement of Theorem \ref{th_A} for the simply laced root systems and also for $\Phi=\mathsf F_4$.

%\subsection{The simply laced case and $\Phi=\mathsf F_4$}
By a theorem of Carter--Keller--Paige (see \cite{CKP}, (2.4)) (rewritten and explained
by Morris \cite{Mor}),
bounded generation for groups of type $\mathsf A_{l}$, $l\ge 2$, holds for all Dedekind rings $R$ in number fields $K$, with a bound depending on $l$ and also on the degree $d$ of $K$. But since for all degrees $d\ge 3$
the existence of uniform bound already follows from
Theorem B, we only need to take maximum of that, and
the universal bound for $d=2$.
\par
Combining this result with the subsequent work of Trost~\cite{trost2} on the function field case,
one obtains
the following result, see~\cite[Theorem~4.1]{trost2}.

\begin{lm}\cite{trost2}
%%\phantom{p}\quad
\label{lem:trost2}
For each $l\ge 2$, there exists a constant $L=L(l)\in\mathbb N$ such that for any Dedekind ring of arithmetic type $R$, any element in $\mathrm{G}_{\mathrm{sc}}(\mathsf A_{l},\,R)$ is a product of at most $L$ elementary root unipotents.
\end{lm}

In fact, in the sequel we only need the special case
of the above result pertaining to $\SL(3,R)$, which
corresponds to $\mathsf A_2$. Indeed, by stability arguments one has
$L(l)\le L(l-1)+3l+1$ for all $l\ge 2$, so that all
$L(l)$, $l\ge 3$, can be expressed in terms of the
constant $L(2)$. In the function case %Alexander
Trost \cite{trost2} gave the estimate $L(2)\le 65$.
No such explicit estimate is known in the number case.

Now we are in a position to get a particular case of Theorem \ref{th_A}.

\begin{maintheorem}
%%\phantom{p}\quad
\label{infinite2}
Let $\Phi$ be simply laced of rank $\ge 2$ or
$\Phi=\mathsf F_4$, and $R$
be any Dedekind ring of arithmetic type. Then
$G_{\mathrm{sc}}(\Phi,R)$ is a product of at most $L=L(2)N$ elementary unipotents.
\end{maintheorem}

\begin{proof}
Since the fundamental root systems of the simply laced
systems and $\F_4$ are covered by copies of $\A_2$, one can take $\Delta_i=\mathsf A_2$ in Lemma \ref{tavgen}
and then apply Lemma \ref{lem:trost2} to the $\mathsf A_2$ case.
\end{proof}

%\par
%Since the fundamental systems of the simply laced
%systems and $\F_4$ are covered by copies of $\A_2$,
%combining \ref{tavgen} with \ref{trost}

Thus, in addition to Theorem B, we obtain
another stronger form of Theorem A, now without
the assumption that $R^*$ is infinite, but only in the
special case of simply laced systems of rank $\ge 2$ and
$\mathsf F_4$.
The bound here is very rough, since $L(2)$ is the
number of {\it elementary\/} factors, the number of
unitriangular ones can be much smaller. Also, the
use of stability arguments allows one to get much better bounds,
of the type $L=L(2)+M$, with $3N\le M\le 4N$,
where some multiple of $N$ occurs as a summand,
not as a factor.

%%%%%%%%%%%%%%%%%%%%%%%%

\section{Stability of K$_1$-functor and flipping long
and short roots} \label{stability}

%\subsection{The use of stability}

Another  way to reduce bounded generation of $G(\Phi,R)$ to bounded generation of  $G(\Delta,R)$, where $\Delta \subset \Phi $, is called surjective stability of K$_1$-functor. Recall that K$_1(\Phi,R)=G(\Phi,R)/E(\Phi,R)$.  Let the root embedding $\Delta \subset \Phi$ be given.  Surjective stability of K$_1$-functor tells us that
$G(\Phi,R)=E(\Phi,R)G(\Delta,R)$, see, e.g. \cite{stein}, \cite{Pl2}.  Moreover, it provides a reduction from $G(\Phi,R)$ to $G(\Delta,R)$ by a bounded number of steps, with a bound depending on $R$ and the root embedding.  Here is the main observation we use (see \cite{KPV}).

\begin{lm}\label{reduction}
Let $R$ be a  Dedekind ring of arithmetic type. Then (uniform) bounded generation of the groups $G(\Phi,R)$, $\Phi\neq \mathsf C_2$,
follows from (uniform) bounded generation of the group $G(\mathsf A_2,R)$.
\end{lm}

We illustrate Lemma \ref{reduction} by two examples of the Chevalley groups
of types $\Phi={\mathsf B}_l$, $l\ge 3$, and
$\Phi=\mathsf G_2$ with explicit bounds for reduction.
Of course, by stability arguments one can assume that $G(\mathsf B_l,R)$ is already reduced to $G(\mathsf B_3,R)$.

\begin{prop}\label{g2}
Let $R$ be a Dedekind ring and assume that any element of $\mathrm{G}(\mathsf A_2,\,R)$ is a product of at most $L$ elementary root unipotents. Then any element of $\mathrm{G}(\mathsf G_2,\,R)$ is a product of at most $L+20$ elementary root unipotents.
\end{prop}
\begin{proof} From \cite[Proposition~4.3]{KPV} we obtain the universal
bound $L(2)+20$ for the elementary generation
of $\mathrm E_{\mathrm{sc}}(\mathsf G_2,\,R)$ over all Dedekind rings of arithmetic type.
\end{proof}

%A similar result for $\Phi=\mathsf B_3$ is also essentially
%proven in~\cite[Section~6.2]{KPV}. The situation
%is only marginally more complicated in view of the
%fact that the vector representation used there furnishes
%not the simply connected group but the adjoint one.
%% [Kunyavski\u\i--Plotkin--Vavilov]
%Of course, the case $\Phi={\mathsf B}_l$, $l\ge 3$, is reduced to $\Phi=\mathsf B_3$ %either by Lemma \ref{tavgen}
%in conjunction with Lemma \ref{trost}, or again
%thanks to  surjective stability for $K_1(B_{l-1},B_l)$.

%%%%%%%%%%%%%%%%%%%%%%%%%%%
%%%%%%%%%%%%%%%%%%%%%%%%%%
%%%%%%%%%%%%%%%%%%%%%%%%%%%%%%%%

%Here we prove Theorem~A for the Chevalley groups
%of types $\Phi={\mathsf B}_l$, $l\ge 3$, and
%$\Phi=\mathsf G_2$.
%\par
%The following result is essentially \cite{KPV}, Proposition 4.3. In particular it implies the universal
%bound $L(2)+20$ for the elementary generation
%of $\mathrm E(\mathsf G_2,R)$ over all Dedekind rings of arithmetic type.

%A similar result for $\Phi=\mathsf B_3$ is also essentially
%proven in~\cite[Section~6.2]{KPV}. The situation
%is only marginally more complicated in view of the
%fact that the vector representation used there furnishes
%not the simply connected group but the adjoint one.
%% [Kunyavski\u\i--Plotkin--Vavilov]
%Of course, the case $\Phi={\mathsf B}_l$, $l\ge 3$, is reduced $\Phi=\mathsf B_3$ either by \ref{tavgen}
%in conjunction with \ref{trost}, or again by stability.

\begin{prop}
\label{b3}
Let $R$ be a Dedekind ring and assume that any element
of $\mathrm{G_{sc}}(\mathsf A_2,\,R)$ is a product of
$L$ elementary root unipotents. Then any element of $\mathrm{E_{ad}}(\mathsf B_3,\,R)$ is a product of at most $L+31$ elementary root unipotents.
\end{prop}

\begin{proof}
First, observe that~\cite[Lemmas~6.3 and 5.1]{KPV} are valid for any Dedekind ring $R$ (although they are formally stated under the assumption $R=\mathbb F_q[t]$).

By~\cite[Lemma~6.3]{KPV}, each element $x\in\mathrm{E_{ad}}(\mathsf B_3,\,R)$ is a product of an image of $y\in\mathrm{G_{ad}}(\mathsf B_2,\,R)$ and at most $21$ elementary root unipotents. However, since the image of $y$ in $\mathrm{G_{ad}}(\mathsf B_3,\,R)$ is elementary and, in particular, lies in the kernel of the spinor norm, we conclude that $y$ itself lies in the kernel of the spinor norm~\cite[Proposition~3.4.1]{bass}, and therefore $y$ is the image of some $z\in\mathrm{G_{sc}}(\mathsf B_2,\,R)$~\cite[(3.3.4)]{bass}.

Next, by~\cite[Lemma~5.1]{KPV}, $z$ is equal to a product of the image of some $w\in\mathrm{G_{sc}}(\mathsf A_1,\,R)$ and at most $10$ elementary root unipotents (where $\mathsf A_1\subset\mathsf B_2$ is the inclusion on long roots). Therefore $x$ is the product of the image of $w$ in $\mathrm{G_{ad}}(\mathsf B_3,\,R)$ and at most $31$ elementary root unipotents.

However, since the inclusion $\mathsf A_1\subset\mathsf B_3$ factors through $\mathsf A_2$, we conclude that $x$ is a product of an image of some element from $\mathrm{G_{sc}}(\mathsf A_2,\,R)$ and at most $31$ elementary root unipotents. The claim follows.
\end{proof}

\begin{cl}
\label{kpv}
%In the notation of Theorem~\ref{trost},
For any Dedekind ring of arithmetic type $R$, any element of $\mathrm{G_{sc}}(\mathsf B_3,\,R)$ is a product of at most $L(2)+41$ elementary root unipotents.
\end{cl}
\begin{proof}
Any element of $\mathrm{G_{sc}}(\mathsf B_3,\,R)$ is elementary by Lemma~\ref{e=g}, and therefore its image in $\mathrm{G_{ad}}(\mathsf B_3,\,R)$ is a product of at most $L(2)+31$ elementary root unipotents by Proposition~\ref{b3}. However, $\mathrm{G_{sc}}(\mathsf B_3,\,R)$ is a central extension of $\mathrm{G_{ad}}(\mathsf B_3,\,R)$ with the kernel cyclic of order $2$. The generator of the kernel comes from $\mathrm{G_{sc}}(\mathsf D_3,\,R)$~\cite[(3.4)]{bass}, where it can be expressed as a product of at most $10$ elementary root unipotents by~\cite[Theorem~7.2.12]{hom}.
\end{proof}

%%%%%%%%%%%%%%%%%%%%%%%%

\subsection{The case of $\Sp(2l,R)$, $l\ge 3$} Thus,
we are left with the analysis of the the symplectic groups
$\Sp(2l,R)$, $l\ge 2$. Quite amazingly, the results
of \cite{KPV} and \cite{trost2} allow to reduce
$\Sp(6,R)$ to $\SL(3,R)$ as well.
%Actually, in the
%special case $R=\Bbb Z$
As mentioned in \cite{KLPV}, the idea of such a reduction
was contained already in Zakiryanov's thesis, see \cite{Zak}.
%but we have not realised this fact before
%rediscovering the same idea in the general case in
%March 2023\footnote{It is wrongly claimed in
%\cite{Zak} that $\Sp(4,\Bbb Z)$ is not boundedly
%generated. As a result this work has not been given
%the credit it deserves. In particular, we should have cited it in the historical survey of \cite{KPV}.}.
Of course, as above,
the case $\Phi={\mathsf C}_l$, $l\ge 3$, is
immediately reduced to $\Phi=\mathsf C_3$ %either by \ref{tavgen} and \ref{trost}, or again
by stability.

\begin{prop}
\label{c3}
Let $R$ be a Dedekind ring and assume that any element
of $\mathrm{G_{sc}}(\mathsf A_2,\,R)$ is a product of
$L$ elementary root unipotents. Then any element of $\mathrm{E_{sc}}(\mathsf C_3,\,R)$ is a product of at most $L+40$ elementary root unipotents.
\end{prop}

\begin{proof}
As in the case of $\Phi=\mathsf B_3$ we first invoke
\cite[Lemma 6.1]{KPV}, to reduce a matrix from
$\Sp(6,R)$ to a matrix from $\Sp(4,R)$ by 16
elementary transformations. Then we invoke
\cite[Lemma 5.1]{KPV}, to reduce a matrix  from
$\Sp(4,R)$ to a matrix from $\Sp(2,R)=\SL(2,R)$
in {\it long\/} root position by 10 elementary transformations. After that we invoke
Lemma \ref{lem:root} %\cite[Lemma 3]{CaKe}
to get a square in the
non-diagonal position by 4 elementary transformations
in the number case or %\cite[Lemma 3.1]{trost2}
to do the same in the function case by 3
elementary transformations. Now, we can invoke
\cite[Lemma 5.15]{KPV} to move such a matrix in
the long
root fundamental position to a matrix in the short
root fundamental position by 10 elementary transformations. At this stage we can apply
Lemma \ref{lem:trost2} to the {\it short\/} root
$\widetilde{\mathsf A}_2\le \mathsf C_3$, which
gives us $\le 16+10+4+10+L$ elementary moves
in all cases.
\end{proof}

%\begin{rk} The case of bounded generation for $G(D_5\to E_6,R)$ and $G(E_6\to E_7,R)$ is elaborated by stability means  in Theorem 1.1.,  \cite{Gvoz23-1} (cf.\cite{Gvoz22}). \end{rk}

So, we have two remaining tasks. First of all, we have to prove Theorem A in the $\mathsf C_2=\Sp(4,R)$ case,
which is not covered by our previous considerations. Second, we want to make the number $L(2)$, which is a crucial constituent in all estimates,
as small as we can.

%%%%%%%%%%%%%%%%%%%%%%%%

\section{Swindling lemma} \label{sec:sw}

We  concentrate now on minimizing estimates for bounded generation. As we know, this problem depends severely on the number of moves which are necessary in order to move any matrix from $\SL(3,R$), where $R$ is a Dedekind ring of arithmetic type, to the identity matrix.

In this section we establish what Nica \cite{Nica} calls
``swindling lemma'', which is essentially a very weak
form of multiplicativity of Mennicke symbols,
sufficient for our purposes and cheaper than the form
used in \cite{CaKe} in terms of the
number of elementary moves. For the symplectic
case such a lemma in full generality is already
contained in \cite{KPV}, here we come up with a
reverse engineering version of Nica's lemma
\cite[Lemma 4]{Nica} in the linear case. The proof
itself is organised in the same style as the proofs
in \cite[Section~5.3]{KPV}.

\subsection{Swindling lemma for $\SL(3,R)$}
The following result is essentially \cite[Lem\-ma 4]{Nica}.
Of course, formally Nica assumes that $R$ is a
PID, to conclude that {\it all\/} $s$ have the desired
factorisations. But calculations with Mennicke symbols \cite{BMS}
show that his result holds [at least] for all Dedekind
rings. Below we extract the rationale behind his
proof, to apply it in the only situation we need.
Namely, we stipulate that the desired factorisation of
$s$ does exist.

\begin{lm} \label{lem:sw}
Let $R$ be any commutative ring. Assume that
$$ A=\left(\begin{array}{ccc} a&b&0 \\ sc&d&0  \\
0&0&1 \end{array}\right)
\in\SL(3,R) $$
\noindent
where $s$ admits factorisation $s=s_1s_2$ such that
$$ a\equiv d\equiv 1\pmod{s_1},\qquad
a\equiv d\equiv -1\pmod{s_2}. $$
\noindent
Then $A$ can be transformed to
$$ A=\left(\begin{array}{ccc}
\pm a&-sb&0 \\ c&\mp d&0\\
0&0&-1  \end{array}\right)
\in\SL(3,R) $$
\noindent
by \ $\le 11$ elementary moves.
\end{lm}
\begin{proof}
Let $t_1,t_2\in R$ be such that
$$ a=1+s_1t_1=-1+s_2t_2. $$
Below we use programmers' notation style to describe elementary moves,
keeping the letter $A$ to denote all matrices appearing along the way.
\par\medskip
$\bullet$ {\bf Step 1}
$$ A=At_{31}(s_1)=\left(\begin{array}{ccc} a&b&0 \\ sc&d&0  \\
s_1&0&1 \end{array}\right). $$
\par\medskip
$\bullet$ {\bf Step 2+3}
$$ A=t_{13}(-t_1)t_{23}(-s_2c)A=\left(\begin{array}{ccc} 1&b&-t_1 \\ 0&d&-s_2c  \\
s_1&0&1 \end{array}\right). $$
\par\medskip
$\bullet$ {\bf Step 4}
$$ A=t_{31}(-s_1)A=\left(\begin{array}{ccc}
1&b&-t_1 \\ 0&d&-s_2c  \\
0&-s_1b&a \end{array}\right). $$
\noindent
At this stage we have rolled $s_1$ over the diagonal,
by simultaneously moving the $2\times 2$ matrix
from the NW-corner to the SE-corner. Now we have
to roll over $s_2$ by simultaneously returning
our $2\times 2$ matrix back to the NW-corner.
\par\medskip
$\bullet$ {\bf Step 5+6}
$$ A=At_{12}(-b)t_{13}(t_1+s_2)=\left(\begin{array}{ccc}
1&0&s_2 \\ 0&d&-s_2c  \\
0&-s_1b&a \end{array}\right). $$
\noindent
Now we are in exactly the same position as
we were after the first move, and can start rolling
back.
\par\medskip
$\bullet$ {\bf Step 7+8}
$$ A=t_{21}(c)t_{31}(-t_2)A=\left(\begin{array}{ccc}
1&0&s_2 \\ c&d&0  \\
-t_2&-s_1b&-1 \end{array}\right). $$
\par\medskip
$\bullet$ {\bf Step 9}
$$ A=t_{13}(s_2)A=\left(\begin{array}{ccc}
-a&-sb&0 \\ c&d&0  \\
-t_2&-s_1b&-1 \end{array}\right). $$
\par\medskip
$\bullet$ {\bf Step 10+11}
$$ A=At_{31}(-t_2)t_{32}(-s_1b)=\left(\begin{array}{ccc}
-a&-sb&0 \\ c&d&0  \\
0&0&-1 \end{array}\right). $$
\par
For the other choice of signs one should start rolling
over the other way, say, with moving $A$ to
$At_{32}(s_2)$.
\end{proof}

\begin{rk}
Below we state a {\it stronger\/} form of the swindling
lemma for {\it short\/} roots in $\Sp(4,R)$,
Lemma 5.3 = \cite[Proposition 5.10]{KPV}, where an
{\it arbitrary\/} $s$ is rolled over from $c$ to $b$,
so that one could ask, how is it possible that the
symplectic result is more general than the linear one?
The answer is very easy. What we do here is the linear prototype of the swindling lemma for {\it long\/}
roots in $\Sp(4,R)$, Lemma 5.4 = \cite[Lemma 5.7]{KPV},
where a {\it square\/} $s^2$ is rolled over from
$c$ to $b$. Of course we could do the same here,
but then to apply it we would have to use the deep
arithmetic Lemma \ref{lem:root} on the extraction of square roots
of Mennicke symbols, %\cite[Lemma 3.1]{trost2},
which would increase the number of elementary moves.
\end{rk}

\subsection{Swindling lemma for $\Sp(4,R)$}
In what concerns $\Sp(4,R)$ we keep the notation
and conventions of \cite[Section~5]{KPV}. In particular,
$\Sp(4,R)$ preserves the symplectic form with Gram
matrix
$$ \begin{pmatrix}
0&0&0&1\\ 0&0&1&0\\ 0&-1&0&0\\ -1&0&0&0\\
\end{pmatrix}. $$
\noindent
Further, $\alpha$ and $\beta$ are fundamental roots of
$\rC_2$, the corresponding root unipotents are
$$ x_{\alpha}(\xi)=\begin{pmatrix}
1&\xi&0&0\\ 0&1&0&0\\ 0&0&1&-\xi\\ 0&0&0&1\\
\end{pmatrix},\qquad
x_{\beta}(\xi)=
\begin{pmatrix}
1&0&0&0\\ 0&1&\xi&0\\ 0&0&1&0\\ 0&0&0&1\\
\end{pmatrix}, $$
\noindent
while $x_{-\alpha}(\xi)$ and $x_{-\beta}(\xi)$ are
their transposes. Together they generate the
elementary symplectic group $\Ep(4,R)$ which for Dedekind rings of arithmetic type coincides with
$\Sp(4,R)$.
\par
There are two natural embeddings of $\SL(2,R)$
into $\Sp(4,R)$,
the {\it short root\/} embedding $\phi_{\alpha}$
$$ \phi_{\alpha}
\begin{pmatrix}1&\xi\\0&1\\ \end{pmatrix}
=x_{\alpha}(\xi),\qquad
\phi_{\alpha}
\begin{pmatrix}1&0\\\xi&1\\ \end{pmatrix}
=x_{-\alpha}(\xi), $$
\noindent
and the {\it long root\/} embedding $\phi_{\beta}$
$$ \phi_{\beta}
\begin{pmatrix}1&\xi\\0&1\\ \end{pmatrix}
=x_{\beta}(\xi),\qquad
\phi_{\beta}
\begin{pmatrix}1&0\\\xi&1\\ \end{pmatrix}
=x_{-\beta}(\xi), $$
\noindent
and unlike groups of other types in the symplectic case they behave very differently.
\par
The following swindling lemma for the short root embedding that we use in the sequel seems to be
stronger than the linear swindling lemma. But this is because morally the Mennicke symbol constructed via
$\phi_{\alpha}$ is the square root of the Mennicke
symbol constructed via $\phi_{\beta}$. At the same
time, stability reduction, see, for instance, \cite[Lemma 5.1]{KPV}, reduces a symplectic matrix to an element
of $\SL(2,R)$ in the {\it long\/} root embedding. Thus, to be able to use this [seemingly] stronger form of
swindling, we should be able to extract square roots of Mennicke symbols anyway.

\begin{lm}\cite{KPV}\label{sw_sh}
Let $a,b,c,d,s\in R$, $ad-bcs=1$  and, moreover,
$a\equiv d\pmod s$. Then
$$ \phi_{\alpha}
\begin{pmatrix} a&b\\ cs&d\\ \end{pmatrix}=
\begin{pmatrix}
a&b&0&0\\ cs&d&0&0\\ 0&0&a&-b\\ 0&0&-cs&d\\
\end{pmatrix} $$
\noindent
%%\quad\text{
can be moved to %%}\quad
$$ \phi_{\alpha}
\begin{pmatrix} d&c\\ bs&a\\ \end{pmatrix}=\begin{pmatrix}
d&c&0&0\\ bs&a&0&0\\ 0&0&d&-c\\ 0&0&-bs&a\\
\end{pmatrix} $$
\noindent
by $\le 26$ elementary transformations.
\end{lm}

We do not use it here, but to put things in the right prospective, let us reproduce the swindling lemma for long roots \cite[Lemma 5.7]{KPV}, on which the proof of Lemma 5.3 hinges, and which is a true analogue of
Lemma \ref{lem:sw} valid for all commutative rings. The number
of moves here {\it seems\/} to be smaller than in
Lemma \ref{lem:sw} because here we do not return the element
of $\SL(2,R)$ to the initial position, but leave it in the
other embedding (to later return it to the same
{\it short\/} root position).

\begin{lm}\cite{KPV}\label{sw_long}
Let $a,b,c,d,s\in R$, $ad-bcs^2=1$  and, moreover,
$a\equiv d\equiv 1\pmod s$. Then
$$ \phi_{\beta}
\begin{pmatrix} a&b\\ cs^2&d\\ \end{pmatrix}
=\begin{pmatrix}
1&0&0&0\\ 0&a&b&0\\ 0&cs^2&d&0\\ 0&0&0&1\\
\end{pmatrix} $$
\noindent %%% \quad\text{
can be moved to %%%}\quad
$$ \phi_{2\alpha+\beta}
\begin{pmatrix} d&-c\\ -bs^2&a\\ \end{pmatrix}=
\begin{pmatrix}
d&0&0&-c\\ 0&1&0&0\\
0&0&1&0\\ -bs^2&0&0&a\\
\end{pmatrix} $$
\noindent
by $\le 8$ elementary transformations.
\end{lm}

%%%%%%%%%%%%%%%%%%%%%%%%

\section{$\SL(3,R)$: function case} \label{sec:sl3f}

Here we prove that in the function case $L(2)\le 44$. This
%the function case
 allows us to calculate {\it explicit\/}
uniform bounds for the
width of all Chevalley groups of rank $\ge 2$, with
the sole exception of $\Sp(4,R)$. This last case
cannot be reduced to $\SL(3,R)$, but can be treated
similarly --- and in fact {\it nominally\/}\footnote{Of course, the difference comes from
the fact that there we use extraction of square roots
of Mennicke symbols. We {\it could\/} do the same
here, getting a slightly shorter proof, with slightly
worse bounds.} easier, since
there we have swindling lemma for short roots
in full generality, Lemma 5.3 = \cite[Proposition 6.10]{KPV}.

With the bound $L\le 65$ the following result was
already established by %Alexander
 Trost
\cite[Theorem 1.3]{trost2}. We use his arithmetic
lemmas, but to derive the bounded generation
adopt the strategy of Nica
\cite{Nica}, with some improvements suggested in
our previous paper \cite{KPV} .

\begin{lm} \label{lem:SL3f}
For any Dedekind ring of arithmetic type $R$ in
a global function field $K$ any element in $\SL(3,R)$
is a product of $L\le 44$ elementary %transvections.
root unipotents.
\end{lm}

%%% \subsection{Elementary width of $\SL(3,R)$}
\begin{proof}
Let, as always, $K$ be a global function field with the
field of constants $\GF{q}$, and $R=\mathcal O_{K,S}$
be any ring of arithmetic type with the quotient field $K$.
\par\medskip
$\bullet$ We start with any matrix $A\in\SL(3,R)$,
and reduce it to a matrix
$$ A=\left(\begin{array}{cc} a&b \\ c&d  \\ \end{array}\right)
\in \SL(2,R)\le\SL(3,R) $$
\noindent
by $\le 7$ elementary moves.
\par\medskip
$\bullet$ Now by Lemma \ref{local} any matrix in
$A\in\SL(2,\mathcal O_{K,S})$ can be reduced to a
matrix $A\in\SL(2,\mathcal O_{K})$
 at the cost of $\le 3$ elementary moves.
Thus, we can from the very start assume that
$A\in\SL(2,\mathcal O_{K})$, in other words that
$R=\mathcal O_{K}$ is precisely the ring
of integers of $K$.
\par\medskip
$\bullet$
Using a version of Dirichlet theorem (= Kornblum--Landau--Artin theorem in the function case) on
primes in arithmetic progressions %or \cite[Lemma 3.1]{trost2}
we can assume that $bR$ is a prime ideal
at the cost of 1 elementary move. %--- a reference??
%--- Trost needs 2 elementary moves!
\par
Now Lemma \ref{lem:unram} %\cite[Lemma 3.3]{trost2}
implies that there exists
$c\in R$ such that $bc\equiv -1\pmod a$ and
$\delta(b)$ and $\delta(c)$ are coprime. The first
of these conditions guarantees the existence of $d\in R$ such that $ad-bc=1$. Since modulo the root subgroup
$X_{21}=\big\{ t_{21}(\xi),\ \xi\in R\big\}$ a matrix $A\in\SL(2,R)$ only depends on
its first row, by another 1 elementary move we can
assume that the entries of our $A$ themselves have
this last property. At this step we have used 2 %--- or 3??---
elementary moves.
\par\medskip
$\bullet$ Let $u,v\in\Nat$ be such that
$u\delta(b)-v\delta(c)=1$. It follows that
$$ \begin{pmatrix} a&b\\ c&d\\ \end{pmatrix}=
{\begin{pmatrix} a&b\\ c&d\\ \end{pmatrix}}^{u\delta(b)}\cdot
\begin{pmatrix} a&b\\ c&d\\ \end{pmatrix}^{-v\delta(c)} $$
\par\noindent
and we reduce the factors independently.
\par
With this end we proceed exactly as Carter and Keller
do in \cite{CaKe}, and as everybody after them. Namely,
we invoke the Cayley--Hamilton theorem, which asserts
that $A^2=\tr(A)A-I$ so that
$$ A^m=X(\tr (A))I+Y(\tr(A))A, $$
\noindent
where $I$ stands for the identity matrix and $X$, $Y$ are polynomials in $\Int[t]$.
%%% (see Remark \ref{Cheb}  below).
\par
It is well known that $X$ divides $Y^2-1$ or, what
is the same, $Y$ divides $X^2-1$, see the proof
of \cite[Lemma 1]{CaKe}. Since $\Int[t]$ is a unique
factorisation domain, there exists a factorisation
$$ Y=Y_1Y_2,\qquad X\equiv 1\pmod{Y_1},\quad
X\equiv -1\pmod{Y_2}. $$
\noindent
\begin{rk} \label{Cheb}
In fact, $X$ and $Y$ are explicitly known,
{\it morally\/} they are the values of two consecutive Chebyshev polynomials $U_{m-1}$ and $U_m$ at $\tr(A)/2=(a+d)/2$,
which allows one to argue inductively, {\it without swindling\/}. This is essentially the approach taken by
Sergei Adian and Jens Mennicke \cite{AdMe}, only that they are
not aware these are Chebyshev polynomials and have to
establish their properties from scratch.
We do not follow this path here, since it would
require considerably more elementary moves.
\end{rk}
\par%%\medskip
$\bullet$
Thus, for an arbitrary $m$ one has
$$ \begin{pmatrix} a&b\\ c&d\\ \end{pmatrix}^m
=x\begin{pmatrix} 1&0\\0&1\\ \end{pmatrix}+
y\begin{pmatrix} a&b\\ c&d\\ \end{pmatrix}=
\begin{pmatrix} x+ya&yb\\ yc&x+yd\\ \end{pmatrix}, $$
%\par\medskip
\noindent
where $x=X(a+d)$, $y=X(a+d)$. An explicit calculation
shows that
$$ x+ya\equiv a^m\pmod{b}\qquad
\text{and}\qquad  x+ya\equiv a^m\pmod{c}. $$
\par
Substituting $a+d$ into the decomposition $Y=Y_1Y_2$
we get
$$ y=y_1y_2,\qquad\text{where}\quad y_1=Y_1(a+d),\quad y_2=Y_2(a+d). $$
\noindent
By the very definition of $y_1$ and $y_2$
one has
$$ x\equiv 1\pmod{y_1},\qquad
x\equiv -1\pmod{y_2}, $$
\noindent
now as congruences in $R$. Thus,
$$ x+ya\equiv x+yd\equiv 1\pmod{y_1},\qquad
x+ya\equiv x+yd\equiv  -1\pmod{y_2}, $$
\noindent
and we are in a position to apply swindling, as
stated in Lemma \ref{lem:sw}.
\par\medskip
$\bullet$ Now, applying Lemma \ref{lem:sw} we reduce
$$ A^m=\begin{pmatrix} x+ya&yb\\ yc&x+yd\\ \end{pmatrix} $$
\par\smallskip\noindent
to either
$$ B=\begin{pmatrix} x+ya&y^2b\\ c&x+yd\\ \end{pmatrix} $$
\par\smallskip\noindent
or
$$ C=\begin{pmatrix} x+ya&b\\ y^2c&x+yd\\ \end{pmatrix} $$
\par\smallskip\noindent
depending on whether we argue modulo $c$ or modulo $b$, both in $\le 11$ elementary moves.
\par\medskip
$\bullet$ In the first case, $A^m$,
$m=-v\delta(c)$, by one appropriate row transformation we get
$$ t_{12}(*)B=\begin{pmatrix}
\big(a^{\delta(c)}\big)^{-v}&*\\ c&x+yd\\ \end{pmatrix}, $$
\noindent
where $a^{\delta(c)}$, and hence
$\big(a^{\delta(c)}\big)^{-v}$, is congruent
to an element of $\GF{q}$ modulo $c$. Thus,
changing the parameter of the elementary move,
we may from the very start assume that
$$ t_{12}(*)B=\begin{pmatrix}
e&*\\ c&x+yd\\ \end{pmatrix}, $$
\noindent
with $f\in\GF{q}^*$. Two more moves make this
matrix diagonal
$$ t_{21}(-cf^{-1})t_{12}(*)Bt_{12}(*)=h_{12}(f). $$
\noindent
Altogether, we have spent $\le 14=11+3$ elementary moves
to reduce $A^m$ to a semisimple root element in
this case.
\par\medskip
$\bullet$ The analysis of the second case, $A^m$, $m=u\delta(b)$, is similar. As above, by one
appropriate column transformation we get
$$ Ct_{21}(*)=\begin{pmatrix} \big(a^{\delta(b)}\big)^{u}&b\\ *&x+yd\\ \end{pmatrix} $$
\noindent
where $a^{\delta(b)}$, and thus also
$\big(a^{\delta(b)}\big)^{u}$, is congruent
to an element of $\GF{q}$ modulo $c$. Thus,
changing the parameter of the elementary move,
we may from the very start assume that
$$ t_{12}(*)B=\begin{pmatrix}
g&b\\ *&x+yd\\ \end{pmatrix}, $$
\noindent
with $g\in\GF{q}^*$. Two more moves make this
matrix diagonal
$$ t_{12}(*)Bt_{21}(*)t_{12}(-g^{-1}b)=h_{12}(g). $$
\noindent
As above, we have spent $\le 14=11+3$ elementary moves
to reduce $A^m$ to a semisimple root element in
this case as well.
\par\medskip
$\bullet$ As is classically known (see, for instance \cite[Corollary 2.2]{KPV}),
the semisimple root element $h_{12}(fg)=h_{12}(f)h_{21}(g)$ can be expressed as a product of $\le 4$ elementary transformations.
\par\medskip
Altogether this gives us $\le 7+3+2+11+11+3+3+4=44$
elementary moves.
%--- or maybe 3 at the third step,
%%% 45 total. %depending on how good our Dirichlet theorem is.
A reference to Trost would give 65. %Still better by 20 moves, than his own proof.
\end{proof}

\begin{rk} \label{rem:3moves-sl3}
The estimate in Lemma \ref{lem:SL3f} can eventually be slightly improved.
Namely, instead of appealing to Lemma \ref{local} at the second step of the proof,
we could proceed as in \cite[Remark~2.5]{trost2}. More precisely, in our set-up,
there exists an element $x\in K$, transcendental over $\mathbb F_q$, such that
the integral closure of $\mathbb F_q[x]$ in $K$ is isomorphic to $\mathcal{O}_{K,S}$,
see \cite[Example~(ii)]{Ger} or \cite[Proposition~7]{Ros}. More explicitly,
according to Proposition~6 and the subsequent lemma in \cite{Ros},
if $S=\{P_1,\dots ,P_s\}$ and $D=a_1P_1+\dots a_sP_s$ is a positive divisor of
sufficiently large degree, then $D$ appears as the polar divisor $D_{\infty}$ of some $x$,
so that ${\mathrm{div}} (x)=D_0-D_{\infty}$.

This argument allows one to justify the proof of \cite[Lemma~3.1]{trost2} over an arbitrary
$R=\mathcal O_{K,S}$, see \cite{KMR}. However, we do not know whether this is enough to streamline
all steps of our proof, particularly the third one where we use Lemma~\ref{lem:unram}. If yes,
this would save us three elementary moves and give the estimate $L\le 41$.
\end{rk}

%%%%%%%%%%%%%%%%%%%%%%%%

\section{$\Sp(4,R)$: function case} \label{sec:sp4f}

Here we prove that for the group $\Sp(4,R)$
in the function case the uniform bound is $\le 90$.
Modulo Lemma \ref{lem:unram} = \cite[Lemma 3.3]{trost2} it is essentially the
same proof as the one given in \cite[Section~6.4]{KPV},
which from the very start uses extraction of square
roots of Mennicke symbols --- thus,
Lemma \ref{lem:root} = \cite[Lemma 3.1]{trost2}. Since the swindling in short root position established in \cite[Proposition 6.10]{KPV} is already quite general, the {\it only\/} difference
with the proof in \cite{KPV} is the necessity to
invoke Lemma 2.1 to reduce to a matrix with
entries in $\mathcal O_K$, which costs 3 extra moves.

\begin{lm} \label{lem:sp4f}
For any Dedekind ring of arithmetic type $R$ in
a global function field $K$ any element in $\Sp(4,R)$
is a product of $L\le 90$ elementary %transvections.
root unipotents.
\end{lm}

\begin{proof}
Essentially, we argue exactly as in the proof of
Lemma \ref{lem:SL3f}, but now relying on the symplectic versions
of the main lemmas from \cite[Section~5]{KPV}, the
$\SL_2$-part of the argument will be exactly the same,
so we only indicate differences.
\par
As above, we start with a global function field $K$
with the field of constants $\GF{q}$, and any ring of arithmetic type $R=\mathcal O_{K,S}$ therein.
\par\medskip
$\bullet$ We start with any matrix $A\in\Sp(4,R)$,
and reduce it to a matrix
$$ A=\phi_{\beta}
\begin{pmatrix} a&b\\ c&d\\ \end{pmatrix}
=\begin{pmatrix}
1&0&0&0\\ 0&a&b&0\\ 0&c&d&0\\ 0&0&0&1\\
\end{pmatrix}
\in \SL^{\beta}(2,R)\le\Sp(4,R) $$
\noindent
in the {\it long\/} root embedding of $\SL(2,R)$
by $\le 10$ elementary moves, \cite[Lemma 5.1]{KPV}.
\par\medskip
$\bullet$ Now by Lemma \ref{local} any matrix in
$A\in\SL^{\beta}(2,\mathcal O_{K,S})$ can be reduced to a
matrix $A\in\SL^{\beta}(2,\mathcal O_{K})$
 at the cost of $\le 3$ elementary moves so that we
 can from the very start assume that $R=\mathcal O_K$
 is the full ring of integers of $K$.
\par\medskip
The next step does not have analogues for $\SL(3,R)$.
\par\medskip
$\bullet$ Now, being inside $\SL(2,\mathcal O_K)$,
we can invoke Lemma \ref{lem:root} to transform our $A$ to
another
$$ A=\phi_{\beta}
\begin{pmatrix} a&b^2\\ *&*\\ \end{pmatrix}
\in \SL^{\beta}(2,R)\le\Sp(4,R), $$
\noindent
with different $a$ and $b$, at a cost of $\le 3$
elementary moves.
\par\medskip
$\bullet$ Next, we can move such an $A$ to
a matrix of the shape
$$ A=\phi_{\beta}
\begin{pmatrix} a&b^2\\ -c^2&*\\ \end{pmatrix}
\in \SL^{\beta}(2,R)\le\Sp(4,R), $$
\noindent
by $\le 1$ elementary move \cite[Lemma 5.14]{KPV},
which, in turn, can be moved to a short root position
$$ \phi_{\alpha}\begin{pmatrix}
a&b\\ c&d\\
\end{pmatrix}=
\begin{pmatrix}
a&b&0&0\\ c&d&0&0\\ 0&0&a&-b\\ 0&0&-c&d\\
\end{pmatrix}\in \SL^{\alpha}(2,R)\le\Sp(4,R),  $$
\noindent
at a cost of $\le 9$ elementary moves, see \cite[Lemma 5.9]{KPV}. Altogether, this gives us $\le 10$ elementary moves at this step, compare \cite[Lemma 5.15]{KPV}.
\par\medskip
At this stage, we are in the same situation as in the proof of Lemma \ref{lem:SL3f} and can return to its third step repeating
the proof from that point on {\it almost\/} verbatim.
Of course, now we have a stronger and more general
version of swindling, Lemma \ref{sw_sh} instead of Lemma \ref{lem:sw},
but on the other hand, since it involves switching to
long root embeddings, and then back again, it requires many more elementary moves than in the linear case.
Let us list the steps to specify the number of elementary moves.
\par\medskip
$\bullet$ Again, using a version of Dirichlet theorem and
Lemma \ref{lem:unram} =\cite[Lemma 3.3]{trost2}
we can assume that $b$ is prime and $c$ is such
that $\delta(b)$ and $\delta(c)$ are coprime.
This requires $\le 2$ % --- or 3?? ---
elementary moves.
\par\medskip
$\bullet$ After that, it is exactly the same proof as
that of Lemma \ref{lem:SL3f}, with reference to Lemma \ref{sw_sh}
instead of Lemma \ref{lem:sw}, which uses 26 elementary
moves instead of 11 in the linear case.
\par\medskip
$\bullet$ The last three steps are now literally the
same as in the proof of Lemma \ref{lem:SL3f}, adding 3+3+4
elementary moves to reduce $A$ to the identity matrix.
%--- we have forgotten these 3+3 in the proof of Theorem 5.18.
%haven't we? What shall we do now?
\par\medskip
This finishes the proof of Lemma \ref{lem:sp4f}. Altogether we
have used
$$ \le 10+3+3+10+2+26+26+3+3+4=90, $$
elementary moves, as claimed.% --- probably LESS,
%if you look inside.
\end{proof}

\begin{rk} \label{rem:3mves_dp4}
In the spirit of Remark \ref{rem:3moves-sl3}, one can eventually
improve the estimate in Lemma \ref{lem:sp4f} to $L\le 87$ by
circumventing the use of Lemma \ref{local}.
\end{rk}

\begin{rk} \label{corrig}
The 6 elementary moves needed to diagonalize the matrix at the end
of the proof of Lemmas \ref{lem:SL3f} and \ref{lem:sp4f}, have been
forgotten in the proof of \cite[Theorem~5.18]{KPV}. This corrigendum
worsens the estimate in that theorem to $w_E(\Sp(4,\mathbb F_q[t])\le 85$.
\end{rk}

%%%%%%%%%%%%%%%%%%%%%%%%

\section{$\Sp(4,R)$: number case} \label{sec:sp4n}

Thus the only piece that is lacking at this point is a
uniform bound for Dedekind rings $R$ of number
type with finite multiplicative group $R^*$. Since
$G_{\sic}(\Phi,\Int)$, $\mathrm{rk}(\Phi)\ge 2$, are boundedly
generated \cite{Tavgen}, we can henceforth assume
that $R=\mathcal O_K$ is the ring of integers of
an imaginary quadratic field $K$, $|K:\Rat|=2$.
\par
The existence of a uniform elementary width bound
$L=L(2,2)$  for
$\SL(3,R)$, $R=\mathcal O_K$, $|K:\Rat|=2$,
was established by Carter, Keller and Paige \cite{CKP}
in the language of model theory/non-standard analysis\footnote{In fact, they established the
existence of such a uniform bound $L=L(n-1,d)$ for
$\SL(n,R)$, $R=\mathcal O_K$, $|K:\Rat|=d$,
that depends on the rank $n-1$ of the group {\it and\/}
the degree $d$ of the number field.},
and then presented slightly differently, in more
traditional logical terms, by Morris \cite{Mor}.
Observe, though, that their bound %$L(2)$
is uniform but {\it not\/} explicit.
\par
As we know from Sections \ref{support} and \ref{sec:Tavgen}, the existence of a
uniform bound for $\SL(3,R)$ implies the existence of
uniform bounds for all $G_{\sic}(\Phi,\Int)$,
$\mathrm{rk}(\Phi)\ge 2$, with the sole exception
of $\Sp(4,R)$.
\par
However, using the results of Bass, Milnor and Serre
\cite{BMS} the existence of a uniform bound for elementary width
of $\Sp(4,R)$, $R=\mathcal O_K$, $|K:\Rat|=2$,
can be easily derived by exactly the same methods
as in \cite{CKP}, \cite{Mor}. Below we sketch a proof of
the following result.

\begin{lm} \label{lem:sp4n}
There exists a uniform bound $L=L'(2,2)$ such that
the width of all groups $\Sp(4,R)$, where
$R=\mathcal O_K$ is the ring of integers in a
quadratic number field $|K:\Rat|=2$, does not exceed
$L$.
\end{lm}
%%% \begin{proof}
%%% To discuss the proof of this result,
With this end we have to briefly recall parts of its
general context.
\subsection{Bounded generation of ultrapowers}
 First,
recall that being algebraic groups Chevalley groups themselves commute with direct products:
$$ G(\Phi,\displaystyle{\prod_{i\in I}R_i})=
\prod_{i\in I}G(\Phi,R_i). $$
\noindent
However, elementary groups do not, in general, commute with direct  %limits
products which is due to the lack of the uniform
elementary generation. Namely, Wilberd van der Kallen noticed that the quotient
$$ E(\Phi,R)^{\infty}/E(\Phi,R^{\infty}) $$
\noindent
(countably many copies) is precisely the obstruction
to the bounded generation
of $E(\Phi,R)$. This easily ensues from the following obvious
observation. In the case of $\SL(n,R)$ the following
result is
\cite[Theorem 2.8]{CKP}, generalisation to all Chevalley
groups is
immediate.

\begin{lm}
Let $I$ be any index set and $R_i$, $i\in I$, be a
family of commutative rings. Suppose all $E(\Phi,R_i)$
have elementary width $\le L$,
$$ E(\Phi,R_i)=E^L(\Phi,R_i). $$
\noindent
Then the elementary width of
$$ \prod_{i\in I}E(\Phi,R_i)=E\Big(\Phi,\displaystyle{\prod_{i\in I}R_i}\Big) $$
\noindent
%%% $E\Big(\Phi,\displaystyle{\prod_{i\in I}R_i}\Big)$
does not exceed $2LN$. Conversely, the above
equality implies that all $E(\Phi,R_i)$ are uniformly
elementarily bounded.

\end{lm}
\begin{proof}
Take $g_i\in E(\Phi,R_i)$, $i\in I$, and for each $i\in I$ choose an elementary expression of $g_i$ of length $\le L$, say
$$ g_i=x_{\beta(i)_1}(\xi(i)^1)\ldots
x_{\beta(i)_L}(\xi(i)^L)=\prod_{j=1}^L x_{\beta(i)^j}(\xi(i)_j)\in E(\Phi,R_i). $$
\noindent
If an actual expression of $g_i$ is shorter than $L$,
just set the remaining $\beta(i)_j$ to the maximal
root of $\Phi$ and $\xi(i)_j$ to $0$.
\par
Now consider any ordering of roots in
$\Phi=\{\gamma_1,\ldots,\gamma_{2N}\}$ and
form products
$$ u(i)^j=x_{\gamma_1}(\xi(i)^j_{1})\ldots
x_{\gamma_{2N}}(\xi(i)^j_{2N})=\prod_{h=1}^{2N} x_{\gamma_h}(\xi(i)^j_{h})\in E(\Phi,R_i),\qquad
1\le j\le L, $$
\noindent
by the following rule:
$$ \xi(i)^j_h= \begin{cases}
\xi(i)^j,&\text{if $\beta(i)_j=\gamma_h$,}\\
0,&\text{otherwise.}
\end{cases}. $$
\noindent
Then clearly
$$ g_i=u(i)^1\ldots u(i)^L\in E(\Phi,R_i),\qquad i\in I. $$
\par
Thus, every element
$$ g=(g_i)_{i\in I}\in \prod_{i\in I}E(\Phi,R_i) $$
\noindent
can be expressed as $g=u^1\ldots u^L$, where
each of the $L$ factors
$$ u^j=(u(i)^j)_{i\in I}\in E(\Phi,\prod_{i\in I}R_i),
\qquad 1\le j\le L, $$
\noindent
can be expressed as a product of $2N$ elementary
generators
$$ u^j=x_{\gamma_1}\Big((\xi(i)^j_1)_{i\in I}\Big)\ldots
x_{\gamma_{2N}}\Big((\xi(i)^j_{2N})_{i\in I}\Big) $$
\noindent
with parameters in $\displaystyle{\prod_{i\in I}R_i}$.
\par
Conversely, an element $g=(g_i)_{i\in I}$, where
$g_i\in E(\Phi,R_i)$, such that the length of its
components $g_i$ is unbounded, cannot possibly
belong to $E(\Phi,\displaystyle{\prod_{i\in I}R_i})$.
\end{proof}

Since bounded generation is inherited by factors, any
{\bf ultraproduct}
%% $\displaystyle{\prod_{i\in I}}{}$
$R=\prod_{\mathcal U} R_i$ of rings $R_i$ for which the elementary groups $E(\Phi,R_i)$ are uniformly
boundedly generated enjoys the property that
$E(\Phi,R)$ is boundedly generated. In particular,
this applies to {\bf ultrapowers} ${}^*R$, also known
as nonstandard models of $R$. In other words, we
have  the following result.

\begin{lm}
Bounded elementary generation of the elementary
group $E(\Phi,R)$ is equivalent to the equality
$$ {}^*E(\Phi,R)=E(\Phi,{}^*R), $$
\noindent
for all non-standard models ${}^*R$ of $R$.
\end{lm}
\begin{proof}
By the remark preceding the statement of this lemma,
we only need to check the inverse implication.
Denote by $F$ the Fr\'echet filter on $ \Nat$.
Assume that $E(\Phi,R)$  is not boundedly generated,
or, what is the same, there exists a sequence
$g_i\in E(\Phi,R)$, $i\in\Nat$, of matrices from
$E(\Phi,R)$ with infinitely growing lengths. Then the element $g=(g_1,g_2,g_3,\dots )$ has infinite length in $E(\Phi,R)^{\infty}$ and, by the very
definition of $F$, also in $E(\Phi,R)^{\infty}/F$.
In other words, $g\notin E(\Phi,R^{\infty}/F)$.
\par
Since
$F$ is the intersection of all non-principal ultrafilters,
there exists a non-principal ultrafilter $\mathcal U$
such that the image of $g$ in the ultrapower $^{\ast}E(\Phi,R)=E(\Phi,R)^{\infty}/{\mathcal U}$
%%% still has an infinite length. Hence $g$
does not belong to $E(\Phi,^{\ast}\!R)=
E(\Phi,R^{\infty}/{\mathcal U})$. Thus,
for this particular ${\mathcal U}$ we have
$^{\ast}E(\Phi,R) \neq E(\Phi,^{\ast}\!R)$.
\end{proof}
\begin{rk}
Assuming the Continuum Hypothesis, all ${}^\ast\!R$
are isomorphic and one has to require the equality
$^{\ast}E(\Phi,R)=E(\Phi,^{\ast}\!R)$ for {\it one\/}
non-standard model. Otherwise, there are $2^{2^{\aleph_0}}$ ultrafilters that lead to
non-isomorphic ${}^\ast\!R$, and to be on the safe
side one has to stipulate this equality for all of them.
\end{rk}

%% Namely, they observed that one does not have
%% to impose that ${}^*E(\Phi,R)=E(\Phi,{}^*R)$,

\subsection{Congruence subgroup problem for
non-standard models} However, Carter, Keller and Paige  \cite{CKP}
made this observation quite a bit more precise. Namely (see \cite[2.1]{CKP} or
\cite[Lemma 2.29]{Mor}):

\begin{lm} \label{lem:fin-ind}
Bounded elementary generation of the elementary
group $E(\Phi,R)$ is equivalent to the condition
$$ E(\Phi,{}^*R)\quad \text{has a finite index in}
\quad {}^*E(\Phi,R), $$for all non-standard models ${}^*R$ of $R$.
%\noindent
\end{lm}
 %using transfer principle
In fact, they proved that bounded elementary generation of the elementary group $E(\Phi,R)$
is equivalent to the almost positive solution of the
congruence subgroup problem for all non-standard
models ${}^*R$ of $R$. %, see \cite[2.1]{CKP} or
%\cite[Lemma 2.29]{Mor}.
%\par

%In fact, an even weaker condition suffices, stated
%in terms of non-standard integers.
\par
More precisely, \cite[2.3]{CKP} and \cite{Mor} apply the whole
machinery not just to the bounded generation of
$\SL(n,R)$, but also to the bounded generation of
$E(n,R,\mq)$ in terms of the conjugates of
elementary generators of level $\mq$. They consider  groups
$ \SL(n,R,\mq)/E(n,R,\mq)$ which are isomorphic to quotients of universal Mennicke groups $C(\mq)$ %which are isomorphic to quotients $ \SL(n,R,\mq)/E(n,R,\mq)$
 for all nonzero
ideals $\mq\unlhd R$ and restate bounded
elementary generation of $E(n,R,\mq)$ as the almost positive
solution of the congruence subgroup problem for
$\SL(n,{}^*R)$.
\par
Recall that  we need a universal bound that depends only
on the root system $\Phi$  and the degree of $K$. Since we reduce the problem to the congruence subgroup problem for
$\SL(n,{}^*R)$, we are very close to that.

\subsection{Universal bound}
Before going to $\Sp (4,R)$ we shall recall one more principal invention of Carter, Keller and Paige \cite{CKP}.
We need to pass from the
ring of integers $R=\mathcal O_K$ in an algebraic number field to the ring ${}^*R$.
The ring ${}^*R$ is a non-standard  model of $R$, that is ${}^*R$ equals to an ultrapower $\prod_\mathcal U R$ along the  ultrafilter $\mathcal U$. The good point is that thanks to \L{}o\'s's theorem \cite[Theorem 4.1.9]{ChKe} ultrapowers keep first order properties of structures unchanged. The bad point is that other properties do not survive, and in many senses ${}^*R$ is far away from the ring of integers $R$.

%The ring   is {\it uniformly} elementarily boundedly
%generated,

With this end,
Carter, Keller and Paige \cite{CKP} introduce arithmetic
conditions $\mathrm{Gen}(t,r)$ and
$\mathrm{Exp}(t,s)$ on a ring $R$ that depend on
natural parameters $r,s$ and $t$, which are too technical to describe them here in full. Morally,
$\mathrm{Gen}(t,r)$ allows to uniformly bound the number of generators of the abelian groups $C(\mq)$,
while $\mathrm{Exp}(t,s)$ allows to uniformly
bound their exponent. Besides, \cite{CKP} constantly used the fact that the stable rank of $R$ is 1.5.
\par
The importance of these conditions consists in the
following %three
pivotal observations.
First of all, conditions $\mathrm{Gen}(t,r)$,
$\mathrm{Exp}(t,s)$ and $\sr(R)=1.5$ are stated in the first order
language of ring theory, see \cite[2.2]{CKP} or
\cite[Sections~3A and~3B]{Mor}. Hence, the characteristic property of ultraproducts imply:

%Since by \L{}o\'s's theorem \cite[Theorem 4.1.9]{ChKe} ultrapowers keep first order properties unchanged, we have
%the following equivalence.

\begin{lm}
A commutative ring $R$ satisfies conditions
$\mathrm{Gen}(t,r)$ and $\mathrm{Exp}(t,s)$
with specific parameters if and only if ${}^*R$
satisfies these conditions with the same parameters. Besides, $\sr({}^*R)$ is 1.5.
\end{lm}

Most importantly, these
conditions allow to bound {\it uniformly\/} the universal Mennicke groups $C(\mq)$ for all ideals $\mq\unlhd R$
--- and thus to get the finite congruence kernel of
$G(\Phi,{}^*R)$. Indeed, the main (difficult!) step
in obtaining a uniform bound in the number case
%% (depending on rank $\mathrm{rk}(\Phi)$ and
%% degree $d=|K:\Rat|$ alone)
is the following result, see \cite[Theorem 1.8]{CKP}
or \cite[Theorem 3.11]{Mor}.

\begin{lm} \label{lem:lin-univ}
Let $r,s,t$ be positive integers, $R$ be an integral
domain subject to the conditions
\par\smallskip
$\bullet$ %%%$1)$
$\sr(R)=1.5$,\qquad
%% \par\smallskip
$\bullet$ %%% $2)$
$\mathrm{Gen}(t,r)$,\qquad
%% \par\smallskip
$\bullet$ %%% $3)$
$\mathrm{Exp}(t,s)$.
\par\smallskip\noindent
Then for all ideals
$\mq$ the universal Mennicke group
$C(\mq)$ is finite and its order is uniformly bounded by $t^r$.
\end{lm}

%Clearly, replacing in the proofs of these results
%all references to \cite[Ch.~II]{BMS} by references
%to the corresponding sections of \cite[Ch.~III]{BMS},
%we get the same uniform bounds for the universal symplectic Mennicke groups $\Cp(\mq)$.

%The following (not less technically involved) result proven by Trost serves as a symplectic analogue
%of Lemma \ref{lem:lin-univ}.

%\begin{lm} \label{lem:sp-univ} \cite[Theorem 3.16]{trost}
%Let $s,t$ be positive integers, $R$ be a commutative ring with 1
%subject to the conditions
%\par\smallskip
%$\bullet$ %%%$1)$
%$\sr(R)=1.5$,\qquad
%% \par\smallskip
%$\bullet$ %%% $2)$
%$\mathrm{Gen}(2,1)$,\qquad
%% \par\smallskip
%$\bullet$
%$\mathrm{Gen}(t,1)$, \qquad
%$\bullet$ %%% $3)$
%$\mathrm{Exp}(t,s)$.
%\par\smallskip\noindent
%Then for all ideals
%$\mq$ the group
%$\Cp(\mq)$ is finite and its order is uniformly bounded by $2t$.
%\end{lm}

Finally, the rings of integers
of the number fields of bounded degree
$|K:\Rat|\le d$ satisfy these conditions for
some values of parameters (which depend on $d$
and which we do not
wish to specify here). This result depends on [a
very strong form of] the Dirichlet theorem on primes in
arithmetic progressions, and a bunch of other
deep arithmetic results. The following lemma
is a [weaker form of the] conjunction of
\cite[Lemmas 4.4 and 4.5]{CKP} or
\cite[Corollary 3.5 and Theorem 3.9]{Mor}.

\begin{lm}
The ring of integers $R=\mathcal O_K$ in an algebraic number field $K$ satisfies $\mathrm{Gen}(t,1)$ for every positive integer $t$
and $\mathrm{Exp}(t,2)$ for some $t$ depending on the degree $d=[K:\Rat]$.
\end{lm}

Hence,  one can take $R={}^*R$ in Lemma \ref{lem:lin-univ} and arrive at the almost positive solution of the congruence problem in $\SL(n,{}^*R)$, as required.

To use the above  fact for $\SL(n,{}^*R)=\SL(n,\prod_\mathcal U R)$ is the same as to use the Compactness Theorem as is in \cite{CKP} or \cite[Theorem 2.7]{Mor}. In any case  (cf. \cite[Theorem 2.4]{CKP} or \cite[Corollary 3.13]{Mor}) we get bounded elementary generation of $\SL(n,R)$ with the bound  that only depends
on the root system $\Phi$  and the degree of $K$.

At this point \cite[2.5]{CKP} use a standard argument from non-standard analysis. Since
the elementary width $w(G)\in{}^*\Nat$ of
$G=\SL(n,R)$ on this class of rings $R$ is internally
defined, everywhere finite, and bounded (by any
infinite natural number), it must attain
maximal value, which is obviously finite (all of them are!).
\par

The proof of the uniform bounded generation for Chevalley groups will be completed if we cover the case of symplectic groups. The standard reasoning says that it is enough to prove the fact for the rank 2 case, that is for $\Sp(4,R)$. The proof basically follows the line depicted above for $G=\SL(n,R)$.

Passing from $\SL(n,R)$ to $\Sp(2l,R)$ ($l\ge 2$), we have to consider
the group $\Cp(\mq):=\Sp(2l,\mq)/\Ep(2l,\mq)$ in place of
$C(\mq)$ (recall that the group $\Cp(\mq)$ is well defined, finite,
and independent of $l$, see \cite[Theorem 12.4 and Corollary 12.5]{BMS}).

%%%%%%%%%%%%%%%%%%%%%%%%%%%%%%%%%%%%%%%%%%%%%
%%%%%%%%%%%%%%%%%%%%%%%%%%%%%%%%%%%%%%%%%%%%%
The case $\Sp(4,R)$ is similar but much more difficult than the one of $\SL(3,R)$. The point is that there are two embeddings
$\rA_1\to \rC_2$ and $\widetilde \rA_1\to \rC_2$, on long and short roots, respectively. This results in a more complicated structure of the universal Mennicke group (cf. \cite[Lemma 13.3]{BMS}).  The decisive role is played by the following quite technical theorem  proven by Trost. It
serves as a symplectic analogue
of Lemma \ref{lem:lin-univ}.

%The following (not less technically involved) result proven by Trost

\begin{lm} \label{lem:sp-univ} \cite[Theorem 3.16]{trost}
Let $s,t$ be positive integers, $R$ be a commutative ring with 1
subject to the conditions
\par\smallskip
$\bullet$ %%%$1)$
$\sr(R)=1.5$,\qquad
%% \par\smallskip
$\bullet$ %%% $2)$
$\mathrm{Gen}(2,1)$,\qquad
%% \par\smallskip
$\bullet$
$\mathrm{Gen}(t,1)$, \qquad
$\bullet$ %%% $3)$
$\mathrm{Exp}(t,s)$.
\par\smallskip\noindent
Then for all ideals
$\mq$ the group
$\Cp(\mq)$ is finite and its order is uniformly bounded by $2t$.
\end{lm}

%% \begin{proof}
%Roughly, the argument is as follows.
%Clearly, $\sr(R)=1.5$ is a first order condition. Indeed,
%it is equivalent to $\sr(R/aR)=1$ for all $a\in R$,
%$a\neq 0$. Thus, one has $\sr({}^*R)=1.5$.
%\par
%This means that surjective stability holds:
%$$ \Sp(2l,{}^*R,{}^*\mq)=\Sp(2,{}^*R,{}^*\mq)
%\Ep(2l,{}^*R,{}^*\mq), $$
%\noindent
%so that the natural Mennicke symbol
%$$ \SL(2,{}^*R,{}^*\mq)\longrightarrow Cp({}^*\mq),\qquad
%(a,b)\mapsto \Big\{{\frac ba}\Big\}\quad \text{(long root embedding)} $$
%\noindent
%is surjective by \cite[Theorem 12.4]{BMS}.
%%%% Recall that $\Ep(4,{}^*R,{}^*\mq)\unlhd\Sp(4,{}^*R)$, see, for instance, \cite{Kopeiko}.
%In particular, the quotient
%$$ \K_1(\rC_l,{}^*R)=\Sp(2l,{}^*R)/\Ep(2l,{}^*R) $$
%\noindent
%is certainly finite, so that Lemma~\ref{lem:fin-ind} implies that
%$\Ep(2l,R)$ is boundedly generated.
%% \end{proof}
%\par\smallskip

%However, we need more than that, we need a
%%%%%%%%%%%%%%%%%%%%%%%%%%%%%%%%%%%%%%%%%%%%%
%%%%%%%%%%%%%%%%%%%%%%%%%%%%%%%%%%%%%%%%%%%%%

This parallelism can now be used to conclude that
$\Sp (4,R)$ is {\it uniformly} elementarily boundedly
generated, with a
{\it universal\/} bound that only depends on %$\Phi$ and
the degree of $K$ (which is equal to 2 in the case under consideration), and is thus
an {\it absolute} constant (see \cite[Theorem~3.20 and Section~3.3]{trost} for details of the proof).
This concludes the proof of Lemma \ref{lem:sp4n}.

%However, this argument only uses that $\SL(n,R)$
%themselves and their sets of elementary generators
%are defined by first-order formulas. Exactly the
%same reasoning applies to any other group functor
%$G$ applied to our class of rings $R$, provided that

%1) the groups $G(R)$ and their generating sets
%$X(R)$ are defined by first order formulas;

%2) each individual group $G(R)$ is boundedly
%generated by $X(R)$.
%\par
%Thus, by Lemma \ref{lem:sp-univ} we can apply the same
%argument to the symplectic groups $\Sp(2l,R)$,
%$l\ge 2$, which, in particular, proves Lemma \ref{lem:sp4n}.
%\par\smallskip

\begin{rk}
%Morris rephrases
The arguments presented above can be rephrased in a more
traditional logical language, in the form of the compactness theorem
of the first order logic, as was done by Morris  \cite[Proposition 1.5]{Mor} and Trost \cite[Theorem~3.1]{trost}. %for a more precise formulation).
%Of course, again he states this result only for
%$\SL(n,R)$, but since any Chevalley group $G_{\sic}(\Phi,R)$ and its set $X$ of elementary generators $x_{\al}(\xi)$,
%$\al\in\Phi$, $\xi\in R$, are described by first
%order formulas, nothing changes in general.
\end{rk}

%\begin{lm} \label{compact}
%Let $\Phi$ be a root system and $\mathcal T$ be
%a set of first-order axioms in the language of ring
%theory. Suppose that, for every commutative ring
%$R$ satisfying the axioms in $\mathcal T$, the
%elementary subgroup $E(\Phi,R)$ has finite index in
%$G(\Phi,R)$. Then, for all such $R$, the elementary
%generators boundedly generate $E(\Phi,R)$.
%\par
%More precisely, there exists a positive integer
%$L=L(\Phi,\mathcal T)$, such that, for all $R$ as
%above, every element of $E(\Phi,R)$ is a product
%of $\le L$ elementary generators.
%\end{lm}

%7777777777777777777777777777777777

%\begin{rk}

%ZHENYA,

%ZDES' NUZHNY KAKIE-TO SLOVA.

%$C_2$ v harakteristike 2, perehod k $x_\xi(1)$, i t.d.

%\end{rk}

%%%%%%%%%%%%%%%%%%%%%%%%%%%%%%%%%%%%%%%%%%%%%%%%%%%%%%%%%%%%%%%%%%%%%%%%%%%%%%

\section{Concluding remarks} %{Remarks and Conjectures}
\label{sec-conj}

Here we mention a couple of eventual generalisations of our results.
\par\smallskip
$\bullet$ In the number case the bounds for
$L(2,2)$ and $L'(2,2)$ are not explicit at all. It seems
that it might be quite a challenge to obtain {\it any\/} explicit bounds. Our impression is that it cannot be easily done at the level of the groups, one should invoke much more arithmetics.
\par\smallskip
$\bullet$ On the other hand, we do not claim that the bounds obtained in the present paper in the function
case are sharp in any sense. It is another, maybe
even a greater challenge to obtain such sharp bounds.
It is usually very hard to estimate width from above,
but still harder to estimate it from below.

\par\smallskip
$\bullet$ Let $\mq\unlhd R$ be an ideal of $R$.
In the present paper we addressed the {\it absolute\/} case $\mq=R$ alone. However, it makes sense to ask similar questions for
the {\it relative\/} case, in other words we believe
there are {\it uniform\/} width bounds for the
true elementary subgroup $E(\Phi,\mq)$ and the relative elementary subgroups $E(\Phi,R,\mq)$ of level
$\mq\unlhd R$, in terms of elementary generators,
or elementary conjugates of level $\mq$.
\par
There are some partial results in this direction for
classical groups, but some of them use larger sets of generators. The results by Tavgen \cite{Tavgen}, Sergei Sinchuk and Andrei Smolensky \cite{SiSm} and by Pavel Gvozdevsky \cite{Gvoz23} use correct sets
of generators, but their bounds are not uniform.
In the function case, Trost \cite{trost3} produced uniform bounds for all types of
irreducible root systems except $\rB_n$ and $\rD_n$.

\par\medskip
\noindent
{\it Acknowledgements.} We are grateful to Nikolai Bazhenov, Sergei Gorchinsky, Pavel Gvozdevsky, Andrei Lavrenov, Alexei Myasnikov, Denis Osipov, Ivan Panin, Victor Selivanov, and Dmitry
Timashev for useful discussions of various aspects of this work [and \cite{KLPV}]. Our special thanks go to Alexander Trost whose thoughtful comments on an earlier version of this paper
allowed us to considerably improve the exposition.

\end{document}